\documentclass[12pt, twoside]{article}
\usepackage{times}
\usepackage{bm}
\usepackage{graphics}
\usepackage{theorem}
\usepackage{graphicx}
\usepackage{amsmath}
\usepackage{latexsym}
\usepackage{amssymb,mathrsfs}
\usepackage[flushmargin]{footmisc}

\theorembodyfont{\itshape}
\newtheorem{thm}{Theorem}[section]
\newtheorem{lem}{Lemma}[section]
\newtheorem{prop}{Proposition}[section]
\newtheorem{coro}{Corollary}[section]

\theorembodyfont{\rmfamily}
\newtheorem{defn}{Definition}[section]{\bf}{\rm}
\newtheorem{assumpt}{Assumption}[section]{\bf}{\rm}

\newtheorem{rem}{Remark}[section]{\itshape}{\rmfamily}

\newenvironment{proof}{\noindent{\it Proof.~~}}{\qed\medskip}
\newenvironment{proofof}[1]{\noindent{\it Proof of #1.~~}}{\qed\bigskip}

\makeatletter
\def\eqnarray{\stepcounter{equation}\let\@currentlabel=\theequation
\global\@eqnswtrue
\global\@eqcnt\z@\tabskip\@centering\let\\=\@eqncr
$$\halign to \displaywidth\bgroup\@eqnsel\hskip\@centering
  $\displaystyle\tabskip\z@{##}$&\global\@eqcnt\@ne 
  \hfil$\;{##}\;$\hfil
  &\global\@eqcnt\tw@ $\displaystyle\tabskip\z@{##}$\hfil 
   \tabskip\@centering&\llap{##}\tabskip\z@\cr}
\makeatother
\makeatletter
    \renewcommand{\theequation}{%
    \thesection.\arabic{equation}}
    \@addtoreset{equation}{section}
  \makeatother

\newcommand{\dm}{\displaystyle}
\newcommand{\qed}{\hspace*{\fill}$\Box$}
\newcommand{\vc}{\bm}
\def\svc#1{\mbox{\boldmath $\scriptstyle #1$}}

\def\presub#1{\hspace{0.05em}{}_{(#1)}\hspace{-0.05em}}
\def\ol#1{\overline{#1}}
\def\wt#1{\widetilde{#1}}
\def\wh#1{\widehat{#1}}
\newcommand{\sfBI}{\mathsf{BI}}
\newcommand{\sfBM}{\mathsf{BM}}

\newcommand{\EE}{\mathsf{E}}
\newcommand{\PP}{\mathsf{P}}

\newcommand{\calF}{\mathcal{F}}
\newcommand{\calG}{\mathcal{G}}

\newcommand{\bbB}{\mathbb{B}}
\newcommand{\bbC}{\mathbb{C}}
\newcommand{\bbD}{\mathbb{D}}
\newcommand{\bbF}{\mathbb{F}}

\newcommand{\bbN}{\mathbb{N}}

\newcommand{\bbZ}{\mathbb{Z}}

\newcommand{\rmd}{{\rm d}}
\newcommand{\rme}{{\rm e}}

\newcommand{\dd}[1]{\if#11 1\!\!1 
\else {\if#1C I\!\!\!C
\else {\if#1G I\!\!\!G 
\else {\if#1J J\!\!\!J 
\else {\if#1S S\!\!\!S
\else {\if#1Z Z\!\!\!Z
\else {\if#1Q O\!\!\!\!Q
\else I\!\!#1
\fi} 
\fi}
\fi}
\fi} 
\fi} 
\fi} 
\fi} 

\makeatother

\begin{document}\thispagestyle{plain} 

\hfill

{\Large{\bf
\begin{center}
Continuous-time block-monotone Markov chains and their block-augmented truncations%
\footnote[1]{This paper is published in {\it Linear Algebra and its Applications}, vol.~514, no.~1, pp.~105--150, 2017.}
\end{center}
}
}

\begin{center}
{
Hiroyuki Masuyama%
\footnote[2]{E-mail: masuyama@sys.i.kyoto-u.ac.jp}
}

\medskip

{\small
Department of Systems
Science, Graduate School of Informatics, Kyoto University\\
Kyoto 606-8501, Japan
}

\bigskip
\medskip

{\small
\textbf{Abstract}

\medskip

\begin{tabular}{p{0.85\textwidth}}
This paper considers continuous-time block-monotone Markov chains
(BMMCs) and their block-augmented truncations. We first introduce the
block monotonicity and block-wise dominance relation for
continuous-time Markov chains, and then provide some fundamental
results on the two notions. Using these results, we show that the
stationary distribution vectors obtained by the block-augmented
truncation converge to the stationary distribution vector of the
original BMMC. We also show that the last-column-block-augmented
truncation (LC-block-augmented truncation) provides the best (in a certain sense) approximation to the stationary
distribution vector of a BMMC among all the block-augmented
truncations. Furthermore, we present computable upper bounds for the
total variation distance between the stationary distribution vectors
of a Markov chain and its LC-block-augmented truncation,
under the assumption that the original Markov chain itself may not be
block-monotone but is block-wise dominated by a BMMC with exponential
ergodicity. Finally, we apply the obtained bounds to a queue with a
batch Markovian arrival process and state-dependent departure
rates.
\end{tabular}
}
\end{center}

\begin{center}
\begin{tabular}{p{0.90\textwidth}}
{\small

\medskip

}
\end{tabular}

\end{center}

\section{Introduction}\label{sec-introduction}

This paper considers continuous-time block-structured Markov chains
characterized by an infinite number of block matrices, such as
GI/G/1-type Markov chains (including M/G/1- and GI/M/1-type ones)
\cite{Bini05,Lato99,Neut89} and level-dependent quasi-birth-and-death
processes (LD-QBDs) \cite{Lato99}.  It is, in general, difficult to
calculate the stationary distribution vectors of such Markov chains. A
simple and practical solution to this problem is to adopt the {\it
  augmented northwest-corner truncations} of the infinitesimal
generators (resp.\ the transition probability matrices) in order to compute
the stationary distribution vectors of continuous-time (resp.\ discrete-time)
Markov chains \cite{Gibs87,Hart12}. Naturally, the stationary
distribution vector obtained by the augmented northwest-corner
truncation is an {\it approximation} to the original stationary
distribution vector.  Therefore, it is important to estimate the error
caused by the augmented northwest-corner truncation.

In fact, such error estimation is facilitated by using the
(stochastic) monotonicity of Markov chains (see, e.g.,
\cite{Dale68}). Indeed, it is shown \cite[Theorem 1]{Gibs87} that the
     {\it last-column-augmented northwest-corner truncation
       (last-column-augmented truncation, for short)} yields the {\it
       best} (in a certain sense) approximation to the stationary
     distribution vector of a discrete-time monotone Markov chain. In
     addition, there have been some studies on the total-variation-distance error bound, i.e., upper bound for the
     total variation distance between the stationary distribution
     vectors of the original Markov chain and its
     last-column-augmented truncation.  Tweedie~\cite{Twee98} assumed
     that the original Markov chain is monotone and geometrically
     ergodic, and then derived a computable total-variation-distance error
     bound. Liu~\cite{LiuYuan10} presented such a bound, assuming the
     monotonicity and polynomial ergodicity of the original Markov
     chain. On the other hand, without the monotonicity, Herv\'{e} and
     Ledoux~\cite{Herv14} developed a total-variation-distance error bound for the stationary
     distribution vector obtained approximately by the
     last-column-augmented truncation of a geometrically ergodic
     Markov chain, though the bound includes the second largest
     eigenvalue of the truncated and augmented transition probability
     matrix. Therefore, Herv\'{e} and Ledoux~\cite{Herv14}'s bound is
     not easy to compute, compared with the bounds presented by
     Tweedie~\cite{Twee98} and Liu~\cite{LiuYuan10}.

We have seen that the monotonicity is useful for the error estimation of
the augmented northwest-corner truncations. However, the monotonicity
is a somewhat strong restriction on block-structured Markov chains.
Thus, Li and Zhao~\cite{LiHai00} introduced the block monotonicity of
discrete-time block-structured Markov chains. The block monotonicity
is an extension of the monotonicity to block-structured Markov
chains. Li and Zhao~\cite{LiHai00} also proved (see Theorem 3.6
therein) that the {\it last-column-block-augmented northwest-corner
  truncation (LC-block-augmented truncation, for short)} yields
the {\it best} approximation to the stationary distribution vector of the
block-monotone Markov chain (BMMC) among all the {\it block-augmented
  northwest-corner truncations (called block-augmented truncations,
  for short)}.  Masuyama~\cite{Masu15-ADV,Masu16-SIAM} presented
computable upper bounds for the total variation distance between the
stationary distribution vectors of the original BMMC and its
LC-block-augmented truncation in the cases where the
original BMMC satisfies the {\it geometric} and {\it subgeometric} drift conditions.  The bounds presented in
\cite{Masu15-ADV,Masu16-SIAM} are the generalization of those in
\cite{Twee98,LiuYuan10}.

The existing results reviewed above are established for discrete-time
BMMCs. These results can be applied to continuous-time Markov chains
with bounded infinitesimal generators by the uniformization technique
\cite[Section 4.5.2]{Tijm03}. As for the continuous-time case, Zeifman
et al.~\cite{Zeif14a} presented an error bound for the periodic
stationary distribution obtained by the truncation of a periodic and
exponentially weakly ergodic non-time-homogeneous birth-and-death
process with bounded transition rates (see also
\cite{Zeif14b,Zeif12}). Hart and Tweedie~\cite{Hart12} provided some
sets of conditions, under which the stationary distribution vectors of
the augmented northwest-corner truncations of a continuous-time
monotone Markov chain converge to the stationary
distribution vector of the original Markov chain.

In this paper, we consider continuous-time block-structured Markov
chains with possibly unbounded infinitesimal generators. We first
provide fundamental results on the block monotonicity and block-wise
dominance relation for continuous-time block-structured Markov
chains. Next, we present the definition of the block-augmented
truncation and LC-block-augmented truncation of continuous-time
block-structured Markov chains. We then prove that the
LC-block-augmented truncation of a BMMC is the {\it best} among all the
block-augmented truncations of the BMMC. We also present computable
total-variation-distance error bounds for the stationary
     distribution vector obtained approximately by the LC-block-augmented truncation of a block-structured Markov chain, under the assumption that the original
Markov chain is block-wise dominated by a BMMC with exponential
ergodicity.  Finally, we apply the obtained bounds to the queue length
process in a queueing model with a batch Markovian arrival process (BMAP)
\cite{Luca91} and state-dependent departure rates.

The rest of this paper is divided into five
sections. Section~\ref{sec-preliminary} introduces basic definitions
and notation. Section~\ref{sec-BMMCs} provides fundamental results
associated with continuous-time BMMCs.  Section~\ref{sec-LBC}
discusses the block-augmented truncations.  Section~\ref{sec-bounds}
presents error bounds for the stationary distribution vector obtained
by the LC-block-augmented truncation. Section~\ref{sec-applications}
applies the error bounds to a queueing model.

\section{Basic definitions and notation}\label{sec-preliminary}

Let $\bbZ_+ = \{0,1,2,\dots\}$, $\bbN = \bbZ_+ \setminus \{0\} =
\{1,2,3,\dots\}$ and $\overline{\bbN} = \bbN \cup \{\infty\}$.
Furthermore, let $\bbZ_+^{\leqslant N} = \{0,1,\dots,N\}$ and
$\bbF^{\leqslant N} =\bbZ_+^{\leqslant N} \times \bbD$ for $N \in
\overline{\bbN}$, where $\bbD=\{1,2,\dots,d\} \subset \bbN$. Note here
that $\bbZ_+^{\leqslant \infty} = \bbZ_+$.  For simplicity, we write
$\bbF$ for $\bbF^{\leqslant \infty}$ and $(k,i;\ell,j)$ for ordered
pair $((k,i),(\ell,j))$.

We define $\vc{I}_d$ as the $d \times d$ identity matrix. We may write
$\vc{I}$ for the identity matrix when its order is clear from the
context. We also define $\vc{O}$ as the zero matrix. Furthermore, let
$\vc{T}_d:=\vc{T}_d^{\leqslant \infty}$ denote
\[
\vc{T}_d
= \left(
\begin{array}{ccccc}
\vc{I}_d & \vc{O} & \vc{O} & \vc{O} & \cdots
\\
\vc{I}_d & \vc{I}_d & \vc{O} & \vc{O} & \cdots
\\
\vc{I}_d & \vc{I}_d & \vc{I}_d & \vc{O} & \cdots
\\
\vc{I}_d & \vc{I}_d & \vc{I}_d & \vc{I}_d & \cdots
\\
\vdots      & \vdots      & \vdots      & \vdots      & \ddots
\end{array}
\right),
\]
and $\vc{T}_d^{\leqslant N}$, $N \in \bbZ_+$, denote the
$|\bbF^{\leqslant N}| \times |\bbF^{\leqslant N}|$ northwest-corner
truncation of $\vc{T}_d$, where $|\,\cdot\,|$ represents the
cardinality of the set between the vertical bars. It is easy to see that
\[
\vc{T}_d^{-1}
= \left(
\begin{array}{ccccc}
\vc{I}_d & \vc{O} & \vc{O} & \vc{O} & \cdots
\\
-\vc{I}_d & \vc{I}_d & \vc{O} & \vc{O} & \cdots
\\
\vc{O} & -\vc{I}_d & \vc{I}_d & \vc{O} & \cdots
\\
\vc{O} & \vc{O} & -\vc{I}_d & \vc{I}_d & \cdots
\\
\vdots      & \vdots      & \vdots      & \vdots      & \ddots
\end{array}
\right),
\]
and that $(\vc{T}_d^{\leqslant N})^{-1}$ is equal to the
$|\bbF^{\leqslant N}| \times |\bbF^{\leqslant N}|$ northwest-corner
truncation of $\vc{T}_d^{-1}$.

We now introduce the block monotonicity and block-wise dominance
relation for probability vectors and stochastic matrices, and provide
the definition of block-increasing column vectors. To this end, we
suppose $N \in \overline{\bbN}$. We then define
$\vc{\mu}=(\mu(k,i))_{(k,i)\in\bbF^{\leqslant N}}$ and
$\vc{\eta}=(\eta(k,i))_{(k,i)\in\bbF^{\leqslant N}}$ as arbitrary
probability vectors with block size $d$. We also define
$\vc{P}=(p(k,i;\ell,j))_{(k,i),(\ell,j)\in\bbF^{\leqslant N}}$ and
$\wt{\vc{P}}=(\wt{p}(k,i;\ell,j))_{(k,i),(\ell,j)\in\bbF^{\leqslant
    N}}$ as arbitrary stochastic matrices with block size $d$.
\begin{defn}\label{defn-block-wise-dominance}
The probability vector $\vc{\mu}$ is said to be block-wise dominated
by the probability vector $\vc{\eta}$ (denoted by $\vc{\mu} \prec_d
\vc{\eta}$) if $\vc{\mu}\vc{T}_d^{\leqslant N} \le
\vc{\eta}\vc{T}_d^{\leqslant N}$.
\end{defn}
\begin{defn}[Definition 1.1 and Proposition 2.1, \cite{Masu15-ADV}]\label{defn-BM}
The stochastic matrix $\vc{P}$ and Markov chains characterized by
$\vc{P}$ are said to be block-monotone with block size $d$ if
$(\vc{T}_d^{\leqslant N})^{-1}\vc{P}\vc{T}_d^{\leqslant N} \ge
\vc{O}$, or equivalently, if
\[
\sum_{m=\ell}^N p(k,i;m,j)
\le \sum_{m=\ell}^N p(k+1,i;m,j),\quad (k,i) \in \bbF^{\leqslant
  N-1},\ (\ell,j) \in \bbF^{\leqslant N}.
\]
The set of block-monotone stochastic matrices with block size
$d$ is denoted by $\sfBM_d$.
\end{defn}
\begin{defn}\label{defn-block-wise-domination}
The stochastic matrix $\vc{P}$ is said to be block-wise dominated by
the stochastic matrix $\wt{\vc{P}}$ (denoted by $\vc{P} \prec_d
\wt{\vc{P}}$) if $\vc{P}\vc{T}_d^{\leqslant N} \le
\wt{\vc{P}}\vc{T}_d^{\leqslant N}$.
\end{defn}
\begin{defn}[Definition 2.1, \cite{LiHai00}]
A column vector $\vc{f}=(f(k,i))_{(k,i)\in\bbF^{\leqslant N}}$ is said
to be block-increasing with block size $d$ if $(\vc{T}_d^{\leqslant
  N})^{-1}\vc{f} \ge \vc{0}$, i.e., $f(k,i) \le f(k+1,i)$ for all
$(k,i) \in \bbZ_+^{\leqslant N-1} \times \bbD$. The set of column
vectors block-increasing with block size $d$ by $\sfBI_d$.
\end{defn}

Finally, we present a basic result on block-monotone stochastic
matrices.
\begin{prop}[Proposition~2.2, \cite{Masu15-ADV}]\label{prop-2}
The following are equivalent:
\begin{enumerate}
\item $\vc{P} \in \sfBM_d$;
\item $\vc{\mu}\vc{P} \prec_d \vc{\eta}\vc{P}$ for any two probability
  vectors $\vc{\mu}$ and $\vc{\eta}$ such that $\vc{\mu} \prec_d
  \vc{\eta}$; and
\item $\vc{P}\vc{f} \in \sfBI_d$ for any $\vc{f} \in \sfBI_d$.
\end{enumerate}
\end{prop}

\section{Block-monotone continuous-time Markov chains}\label{sec-BMMCs}

In this section, we first provide the basic assumption and
characterization of a continuous-time block-structured Markov chain.
We then describe the block monotonicity and block-wise dominance
relation for the infinitesimal generators of continuous-time
block-structured Markov chains. We also present some fundamental
results on the block monotonicity and block-wise dominance relation.

\subsection{Block-structured Markov chains}\label{subsec-BSMC}

Let $\{(X_t,J_t);t\ge0\}$ denote a continuous-time Markov chain with
state space $\bbF^{\leqslant N}$, where $N \in \overline{\bbN}$. Let
$\vc{P}^{(t)}=(p^{(t)}(k,i;\ell,j))_{(k,i),(\ell,j)\in\bbF^{\leqslant N}}$
denote the {\it transition matrix function} of $\{(X_t,J_t);t\ge0\}$,
i.e.,
\begin{equation}
p^{(t)}(k,i;\ell,j) =
\PP(X_t=\ell,J_t=j \mid X_0=k,J_0=i),\quad t \ge 0.
\label{defn-p^{(t)}(k,i;l,j)}
\end{equation}
It is known that
\begin{equation}
\vc{P}^{(t+s)} = \vc{P}^{(t)}\vc{P}^{(s)} = \vc{P}^{(s)}\vc{P}^{(t)},
\qquad t,s \ge 0,
\label{CK-EQ}
\end{equation}
which is called the Chapman-Kolmogorov equation \cite[Chapter 8,
  Section 2.1]{Brem99}.

We assume that $\vc{P}^{(t)}$ is {\it continuous (or standard)}, i.e.,
$\lim_{t\downarrow0}\vc{P}^{(t)} = \vc{I}$ (see \cite[Chapter 8,
  Section 2.2]{Brem99} and \cite[Definition at p.~5]{Ande91}). It then
follows from \cite[Section 1.2, Proposition 2.2]{Ande91} that, for all
$(k,i) \in \bbF^{\leqslant N}$, $q_{(k,i)} := \lim_{t\downarrow0}(1 -
p^{(t)}(k,i;k,i))/t \ge 0$ exists. Although $q_{(k,i)}$ is possibly
infinite, we assume in what follows that
\[
q_{(k,i)}< \infty
\quad \mbox{for all $(k,i) \in \bbF^{\leqslant N}$},
\]
that is, all the states in the state space $\bbF^{\leqslant N}$ are
{\it stable} \cite[Definition at p.~9]{Ande91}. Thus, it follows from
\cite[Section 1.2, Proposition 2.4 and Corollary 2.5]{Ande91} that
$\vc{P}^{(t)}$ satisfies the Kolmogorov forward differential equation
(\ref{eqn-FKM-DEQ}) and the backward differential equation
(\ref{eqn-BKM-DEQ}):
\begin{align}
{\rmd \over \rmd t}\vc{P}^{(t)} &= \vc{P}^{(t)}\vc{Q}, 
\label{eqn-FKM-DEQ}
\\
{\rmd \over \rmd t}\vc{P}^{(t)} &= \vc{Q} \vc{P}^{(t)}, 
\label{eqn-BKM-DEQ}
\end{align}
where $\vc{Q}:=(q(k,i;\ell,j))_{(k,i),(\ell,j)\in\bbF^{\leqslant N}}$ is a
matrix whose elements are all finite, which is given by
\begin{equation}
\vc{Q} = \lim_{t\downarrow0}{\vc{P}^{(t)} - \vc{I} \over t}.
\label{defn-Q}
\end{equation}
Note here that $q_{(k,i)}=|Q(k,i;k,i)| < \infty$ for all
$(k,i)\in\bbF^{\leqslant N}$.  The matrix $\vc{Q}$ is call the {\it
  infinitesimal generator} \cite[Chapter 8, Definition 2.3]{Brem99} of the Markov
chain $\{(X_t,J_t);t\ge0\}$ and the transition matrix function
$\vc{P}^{(t)}$. In general, the infinitesimal generator is a diagonally dominant matrix with nonpositive diagonal and nonnegative off-diagonal elements (see \cite[Section~1.2, Propositions~2.2 and 2.6]{Ande91}). Such a matrix is referred to as 
the {\it $q$-matrix} \cite[Definitions at p.~13 and p.~64]{Ande91}.

For further discussion, we assume that $\vc{Q}$ is
{\it conservative} \cite[Definition at p.~13]{Ande91}, i.e.,
\begin{equation}
\vc{Q}\vc{e} = \vc{0},
\label{eqn-Qe=0}
\end{equation}
where $\vc{e}$ is a column vector of 1's. 

In the rest of this paper, we proceed under Assumption~\ref{assumpt-basic-zero} below, which is a summary of the assumptions made above.
\begin{assumpt}\label{assumpt-basic-zero}
(i)  $\vc{P}^{(t)}$  is continuous; and
(ii) $\vc{Q}$ is stable and conservative.
\end{assumpt}

We now mention an important notion for infinitesimal generators (or
$q$-matrices).  The infinitesimal generator $\vc{Q}$
is said to be {\it regular (or non-explosive)} if the equation
\begin{equation}
\vc{Q}\vc{x} = \gamma\vc{x}, \quad \vc{0} \le \vc{x} \le \vc{e}
\label{eqn-Q-regular}
\end{equation}
has no nontrivial solution for some (and thus all) $\gamma > 0$ (see
\cite[Chapter 8, Theorem 4.4]{Brem99} and \cite[Section 2.2, Theorem
  2.7]{Ande91}). In fact, $\vc{Q}$ is regular if and only if
$\{(X_t,J_t);t\ge0\}$ is a regular-jump process \cite[Chapter 8,
  Definition 2.5]{Brem99}. Furthermore, if $\vc{Q}$ is regular, then
Assumption~\ref{assumpt-basic-zero} holds \cite[Chapter 8, Definition
  2.4 and Theorem 3.4]{Brem99} and thus $\vc{P}^{(t)}\vc{e}=\vc{e}$
for all $t \ge 0$ and $\{\vc{P}^{(t)}\}$ is the unique solution of
both equations (\ref{eqn-FKM-DEQ}) and (\ref{eqn-BKM-DEQ})
\cite[Corollary 2.5 and Theorems 2.2 and 2.7 of Section 2.2 and
  Definition at p.~81]{Ande91}.

\begin{rem}\label{rem-unifomization}
If $\vc{Q}$ is bounded \cite[Section 4.5.2]{Tijm03}, then
\begin{equation*}
\vc{P}^{(t)}= \sum_{m=0}^{\infty}{(\vc{Q} t)^m \over m!}
=\exp\{\vc{Q} t\}.
\end{equation*}
\end{rem}

Finally, we introduce the definition of a stationary distribution
vector (or stationary distribution) of the Markov chain
$\{(X_t,J_t)\}$.

\begin{defn}\label{defn-pi}
Let $\vc{\pi}=(\pi(k,i))_{(k,i) \in \bbF^{\leqslant N}}$ denote a
probability vector such that
\[
\vc{\pi} \vc{P}^{(t)} = \vc{\pi}\quad
\mbox{for all $t \ge 0$}.
\]
The vector $\vc{\pi}$ is called a {\it stationary distribution vector
  (or stationary distribution)} of the Markov chain $\{(X_t,J_t)\}$
and the transition matrix function $\vc{P}^{(t)}$ (see
\cite[Definition at pp.~159--160]{Ande91}).
\end{defn}

\begin{rem}\label{rem-pi}
Suppose that the Markov chain $\{(X_t,J_t)\}$ and thus $\vc{Q}$ are
irreducible. It then holds that $\{(X_t,J_t)\}$ and $\vc{Q}$ are
positive recurrent if and only if there exists a stationary
distribution vector of $\{(X_t,J_t)\}$ and $\vc{P}^{(t)}$ \cite[Section
  5.1, Proposition~1.7]{Ande91}.  Furthermore, it is known
\cite[Section 5.4, Theorem~4.5]{Ande91} that if $\{(X_t,J_t)\}$ is
ergodic (i.e., irreducible and positive recurrent) then its stationary
distribution vector $\vc{\pi}$ satisfies
\begin{eqnarray}
\vc{\pi}\vc{Q} &=& \vc{0},
\label{eqn-piQ=0}
\\
\lim_{t\to\infty}\vc{P}^{(t)} &=& \vc{e}\vc{\pi},
\label{lim-P^{(t)}}
\end{eqnarray}
and (\ref{lim-P^{(t)}}) implies that $\vc{\pi}$ is the unique
stationary distribution vector.
\end{rem}

\begin{rem}\label{rem-regular}
If $\vc{Q}$ is ergodic, then $\vc{Q}$ is regular. Indeed, suppose that
$\vc{Q}$ is ergodic but is not regular. Under this assumption, the
equation (\ref{eqn-Q-regular}) has a nontrivial solution $\vc{x} \ge
\vc{0},\neq \vc{0}$ for some $\gamma > 0$. Pre-multiplying
(\ref{eqn-Q-regular}) by $\vc{\pi}$ and using (\ref{eqn-piQ=0}), we
have $0 = \gamma (\vc{\pi}\vc{x}) > 0$, which yields a
contradiction. Consequently, the ergodicity of $\vc{Q}$ implies that
$\vc{Q}$ is regular.
\end{rem}

\subsection{Block monotonicity and block-wise dominance for infinitesimal generators}

In this subsection, we present the fundamental results on the block
monotonicity and block-wise dominance relation for $|\bbF^{\leqslant
  N}| \times |\bbF^{\leqslant N}|$ infinitesimal generators, where $N \in \overline{\bbN}$.  To this end, we
introduce another Markov chain $\{(\wt{X}_t,\wt{J}_t);t\ge0\}$ with
state space $\bbF^{\leqslant N}$ and infinitesimal generator
$\wt{\vc{Q}}:=(\wt{q}(k,i;\ell,j))_{(k,i),(\ell,j)\in\bbF^{\leqslant
    N}}$. We then define
$\wt{\vc{P}}^{(t)}:=(\wt{p}^{(t)}(k,i;\ell,j))_{(k,i),(\ell,j)\in\bbF^{\leqslant
    N}}$ as the transition matrix function of the Markov chain
$\{(\wt{X}_t,\wt{J}_t)\}$. By definition,
\[
\wt{p}^{(t)}(k,i;\ell,j)
= \PP(\wt{X}_t = \ell, \wt{J}_t = j 
\mid \wt{X}_0 = k, \wt{J}_0 = i),
\]
for $t \ge 0$ and $(k,i;\ell,j) \in \bbF^{\leqslant N} \times
\bbF^{\leqslant N}$.  In addition, we define $\bbB^{\leqslant N}$, $N
\in \ol{\bbN}$, as
\[
\bbB^{\leqslant N}
= (\bbF^{\leqslant N} \times \bbF^{\leqslant N}) 
\setminus \{(k,i;k,i); (k,i)\in \bbF^{\leqslant N}\},
\]
and write $\bbB$ for $\bbB^{\leqslant \infty}$.

We now provide the definition of the
block monotonicity and block-wise dominance. 
\begin{defn}\label{defn-Block-Monotonicity}
The infinitesimal generator $\vc{Q}$ and Markov chains characterized
by $\vc{Q}$ are said to be block-monotone with block
size $d$ (denoted by $\vc{Q} \in \sfBM_d$) if all the off-diagonal elements of
$(\vc{T}_d^{\leqslant N})^{-1} \vc{Q} \vc{T}_d^{\leqslant N}$ are nonnegative, i.e.,
\[
\sum_{m=\ell}^N q(k-1,i;m,j) \le \sum_{m=\ell}^N q(k,i;m,j),
\quad \mbox{$(k,i;\ell,j) \in \bbB^{\leqslant N}$ with $k \in \bbN$}.
\]
\end{defn}
\begin{defn}\label{defn-dominance}
The infinitesimal generator $\vc{Q}$ is said to be
block-wise dominated by $\wt{\vc{Q}}$ (denoted by $\vc{Q} \prec_d
\wt{\vc{Q}}$) if $\vc{Q}\vc{T}_d^{\leqslant N} \le
\wt{\vc{Q}}\vc{T}_d^{\leqslant N}$.
\end{defn}

In what follows, we present five lemmas:
Lemmas~\ref{prop-xi(i,j)}--\ref{lem-Q-4}. For the respective lemmas,
we give the proofs in the case where $N = \infty$ only, which can be
applied to the case where $N < \infty$, with minor modifications.
\begin{lem}\label{prop-xi(i,j)}
If $\vc{Q} \in \sfBM_d$, then (a) $\xi(i,j) := \sum_{\ell\in\bbZ_+^{\leqslant
    N}}q(k,i;\ell,j)$ is constant with respect to $k \in
\bbZ_+^{\leqslant N}$; and (b) $\{J_t;t\ge0\}$ is a Markov chain with state space $\bbD$ and infinitesimal generator $\vc{\Xi} = (\xi(i,j))_{i,j\in\bbD}$.
\end{lem}

\begin{proof}
We first prove statement (a). Since $\vc{Q} \in \sfBM_d$ (see
Definition~\ref{defn-Block-Monotonicity}), we have
\[
\sum_{\ell=0}^{\infty} q(k,i;\ell,j)
\le \sum_{\ell=0}^{\infty} q(k+1,i;\ell,j)
\quad \mbox{for all $(k,i) \in \bbF$ and $j\in\bbD$}.
\]
Combining this and (\ref{eqn-Qe=0}) yields
\[
0 = \sum_{(\ell,j)\in\bbF} q(k,i;\ell,j)
\le \sum_{(\ell,j)\in\bbF} q(k+1,i;\ell,j) = 0
\quad \mbox{for all $(k,i) \in \bbF$},
\]
which implies that, for each $(i,j) \in \bbD^2$,
$\sum_{\ell=0}^{\infty}q(k,i;\ell,j)$ is constant with respect to $k
\in \bbZ_+$, i.e.,
\begin{equation}
\xi(i,j)=\sum_{\ell=0}^{\infty}q(k,i;\ell,j)\quad 
\mbox{for all $(k,i) \in \bbF$ and $j\in\bbD$}.
\label{eqn-xi(i,j)}
\end{equation}
Therefore, statement (a) holds.

Next, we prove statement (b). Let $p_k^{(t)}(i,j) = \PP(J_t = j \mid
X_0 = k, J_0=i)$ for $k \in \bbZ_+$ and $i,j\in\bbD$. It then follows
from (\ref{defn-p^{(t)}(k,i;l,j)}) that
\begin{eqnarray}
p_k^{(t)}(i,j)
&=& \sum_{\ell=0}^{\infty}
p^{(t)}(k,i;\ell,j), \qquad k \in \bbZ_+,\ i,j\in\bbD.
\label{defn-p_k^{(t)}(i,j)}
\end{eqnarray}
Since $\vc{P}^{(t+s)} = \vc{P}^{(s)}\vc{P}^{(t)}$ (see (\ref{CK-EQ}))
and $\vc{P}^{(t)}\vc{e} \le \vc{e}$, we have
\begin{eqnarray}
\lefteqn{
\sum_{(\ell,j)\in\bbF}{ | p^{(t+s)}(k,i;\ell,j) - p^{(t)}(k,i;\ell,j) | \over s}
}
\nonumber
\\
&\le& {1 - p^{(s)}(k,i;k,i) \over s }p^{(t)}(k,i;\ell,j) 
+  \sum_{(\ell',j') \in \bbF \setminus \{(k,i)\}}
{p^{(s)}(k,i;\ell',j') \over s }p^{(t)}(\ell',j';\ell,j)
\nonumber
\\
&\le& {1 - p^{(s)}(k,i;k,i) \over s }
+  \sum_{(\ell',j') \in \bbF \setminus \{(k,i)\}}
{p^{(s)}(k,i;\ell',j') \over s }
\nonumber
\\
&\le& {2\{1 - p^{(s)}(k,i;k,i)\} \over s}
\le 2\,|q(k,i;k,i)|,
\qquad t \ge0,\ s > 0,\ (k,i) \in \bbF,
\label{add-eqn-151108-01}
\end{eqnarray}
where the last inequality holds due to
\cite[Theorem~II.3.1]{Chun67}. It also follows from
(\ref{defn-p_k^{(t)}(i,j)}) and (\ref{add-eqn-151108-01}) that, for $t
\ge0$, $s > 0$, $k \in \bbZ_+$ and $i,j\in\bbD$,
\begin{eqnarray*}
\lefteqn{
{ |p_k^{(t+s)}(i,j) - p_k^{(t)}(i,j)| \over s}
}
\quad &&
\nonumber
\\
&\le&
\sum_{(\ell,j)\in\bbF} { | p^{(t+s)}(k,i;\ell,j) - p^{(t)}(k,i;\ell,j) | \over s}
\le 2q(k,i;k,i).
\end{eqnarray*}
Thus, combining (\ref{eqn-FKM-DEQ}), (\ref{defn-p_k^{(t)}(i,j)}) and
the dominated convergence theorem yields, for $t \ge 0$, $k \in
\bbZ_+$ and $i,j\in\bbD$,
\begin{eqnarray*}
{\rmd \over \rmd t}p_k^{(t)}(i,j)
&=& 
\lim_{s \downarrow0}
{p_k^{(t+s)}(i,j) - p_k^{(t)}(i,j) \over s}
\nonumber
\\
&=& \sum_{\ell=0}^{\infty}
\lim_{s \downarrow0}
{p^{(t+s)}(k,i;\ell,j) - p^{(t)}(k,i;\ell,j) \over s}
\nonumber
\\
&=& \sum_{\ell=0}^{\infty} \sum_{(\ell',j')\in\bbF}  
p^{(t)}(k,i;\ell',j') q(\ell',j';\ell,j)
\nonumber
\\
&=& \sum_{(\ell',j')\in\bbF}  
p^{(t)}(k,i;\ell',j') \sum_{\ell=0}^{\infty}  q(\ell',j';\ell,j).
\end{eqnarray*}
Substituting (\ref{eqn-xi(i,j)}) and (\ref{defn-p_k^{(t)}(i,j)}) into
the above equation, we obtain
\begin{eqnarray*}
{\rmd \over \rmd t}p_k^{(t)}(i,j)
&=& \sum_{j'\in\bbD} \sum_{\ell' \in \bbZ_+} p^{(t)}(k,i;\ell',j') \xi(j',j)
\nonumber
\\
&=& \sum_{j'\in\bbD} p_k^{(t)}(i,j') \xi(j',j),
\qquad t \ge0,\ k \in \bbZ_+,\ i,j\in\bbD.
\end{eqnarray*}
Therefore, 
\begin{eqnarray}
\lefteqn{
\PP(J_t = j \mid X_0 = k, J_0=i)
}
\qquad &&
\nonumber
\\
&=& p_k^{(t)}(i,j)
= [ \exp\{ \vc{\Xi} t\} ]_{i,j},
\qquad t \ge 0,\ k \in \bbZ_+,\ i,j \in \bbD,
\label{eqn-p_k^{(t)}(i,j)}
\end{eqnarray}
where $[\,\cdot\,]_{i,j}$ denotes the $(i,j)$th element of the $|\bbD|
\times |\bbD|$ matrix in the square brackets. In addition, from (\ref{eqn-p_k^{(t)}(i,j)}), we have
\begin{eqnarray}
\PP(J_t = j \mid J_0=i)
&=& \sum_{k=0}^{\infty} [ \exp\{ \vc{\Xi} t\} ]_{i,j}\PP(X_0 = k \mid J_0=i)
\nonumber
\\
&=& [ \exp\{ \vc{\Xi} t\} ]_{i,j},
\qquad t \ge 0,\ i,j \in \bbD.
\label{eqn-161028-01}
\end{eqnarray}
Note here that $\vc{\Xi}$ is a conservative $q$-matrix (i.e., $\vc{\Xi}\vc{e} =
\vc{0}$), because $\vc{\Xi}$ satisfies (\ref{eqn-xi(i,j)}) and $\vc{Q}$
is the infinitesimal generator of the Markov chain $\{(X_t,J_t)\}$.
As a result, (\ref{eqn-161028-01}) shows that $\{J(t);t\ge0\}$ is a Markov
chain with state space $\bbD$ and infinitesimal generator $\vc{\Xi}$.
\end{proof}

\begin{lem}\label{lem-Q}
If $\vc{Q}$ is regular, then the following are equivalent: (a) $\vc{Q}
\in \sfBM_d$; and (b) $\vc{P}^{(t)} \in \sfBM_d$, i.e.,
$(\vc{T}_d^{\leqslant N})^{-1} \vc{P}^{(t)} \vc{T}_d^{\leqslant N} \ge
\vc{O}$ for all $t \ge 0$.
\end{lem}

\proof Before the proof of this lemma, we introduce some symbols. Fix
$n \in \bbN$ arbitrarily and let $t_n = \inf\{t\ge0: X_t \ge
n\}$. Since the Markov chain $\{(X_t,J_t)\}$ is regular,
\begin{equation}
\PP(\lim_{n\to\infty}t_n = \infty) = 1.
\label{lim-P(t_n=infty)}
\end{equation}
Thus, we define $\{(X_t^{\leqslant n},J_t^{\leqslant n}); t\ge0\}$ as a
Markov chain with
state space $\bbF^{\leqslant n}$ such that
\begin{equation}
X_t^{\leqslant n}
= 
\left\{
\begin{array}{ll}
X_t,& 0 \le t < t_n,
\\
n, & t \ge t_n,
\end{array}
\right.
\quad 
J_t^{\leqslant n} = J_t, \quad t \ge 0.
\label{defn-X_t^{<n}}
\end{equation}
We also define $\vc{Q}^{\leqslant n}=(q^{\leqslant
  n}(k,i;\ell,j))_{(k,i),(\ell,j)\in\bbF^{\leqslant n}}$ as the
infinitesimal generator of \break $\{(X_t^{\leqslant n},J_t^{\leqslant
  n})\}$. It then follows that, for $i,j \in \bbD$,
\begin{equation}
q^{\leqslant n}(k,i;\ell,j)
=\left\{
\begin{array}{ll}
q(k,i;\ell,j), & k,\ell \in \bbZ_+^{\leqslant n-1},
\\
\dm\sum_{m=n}^{\infty}q(k,i;m,j), & k\in \bbZ_+^{\leqslant n-1},~\ell = n,
\\
\dm\sum_{m=0}^{\infty}q(n,i;m,j) = \xi(i,j), & k=\ell=n,
\\
0, & \mbox{otherwise}.
\end{array}
\right.
\label{defn-q^{<n}(k,i;l,j)}
\end{equation}
Furthermore, let $p^{\leqslant n;(t)}(k,i;\ell,j)$ denote 
\begin{equation}
p^{\leqslant n;(t)}(k,i;\ell,j)
=
\PP(X_t^{\leqslant n} =\ell,J_t^{\leqslant n}=j 
\mid X_0^{\leqslant n} = k,J_0^{\leqslant n}=i),
\label{defn-p^{<=n;(t)}}
\end{equation}
for $t \ge 0$ and $(k,i), (\ell,j) \in \bbF^{\leqslant n}$.

We are now ready to prove the present lemma. We first prove that
statement (a) implies statement (b). From $\vc{Q} \in \sfBM_d$ (see
Definition~\ref{defn-Block-Monotonicity}), we have
\begin{equation}
\sum_{m=\ell}^{\infty} q(k-1,i;m,j)
\le \sum_{m=\ell}^{\infty} q(k,i;m,j),\quad 
\mbox{$(k,i;\ell,j) \in \bbB$ with $k \in \bbN$}.
\label{defn-Q-in-BM_d}
\end{equation}
From (\ref{defn-q^{<n}(k,i;l,j)}) and (\ref{defn-Q-in-BM_d}), we also
have
\begin{eqnarray}
\lefteqn{
\sum_{m=\ell}^n q^{\leqslant n}(k-1,i;m,j)
}
\quad&&
\nonumber
\\
&=& \sum_{m=\ell}^{\infty} q(k-1,i;m,j)
\le \sum_{m=\ell}^{\infty} q(k,i;m,j)
\nonumber
\\
&=& \sum_{m=\ell}^n q^{\leqslant n}(k,i;m,j),\quad 
\mbox{$(k,i;\ell,j) \in \bbB^{\leqslant n}$ with $k \in \bbN$}.
\label{defn-Q^{<n}-in-BM_d-01}
\end{eqnarray}
The inequality (\ref{defn-Q^{<n}-in-BM_d-01})
implies that all the off-diagonal elements of $(\vc{T}_d^{\leqslant
  n})^{-1}\vc{Q}^{\leqslant n}\vc{T}_d^{\leqslant n}$ are
nonnegative. Thus, we can choose $\sigma_n \in (0,\infty)$ such that
$(\vc{T}_d^{\leqslant n})^{-1}(\vc{I} + \sigma_n^{-1}
\vc{Q}^{\leqslant n}) \vc{T}_d^{\leqslant n} \ge \vc{O}$, which yields
\begin{eqnarray}
\lefteqn{
(\vc{T}_d^{\leqslant n})^{-1} 
\exp\{\vc{Q}^{\leqslant n}t\} \vc{T}_d^{\leqslant n} 
}
~~&&
\nonumber
\\
&=& \sum_{m=0}^{\infty} \rme^{-\sigma_n t}{(\sigma_n t)^m \over m!}
(\vc{T}_d^{\leqslant n})^{-1} 
(\vc{I} + \sigma_n^{-1} \vc{Q}^{\leqslant n})^m 
\,\vc{T}_d^{\leqslant n} \ge \vc{O}, \quad t \ge 0.\qquad
\label{eqn-25}
\end{eqnarray}
It follows from (\ref{eqn-25}) that,
for $(k,i)\in \bbF^{\leqslant n-1}$ and $(\ell,j) \in \bbF^{\leqslant n}$,
\begin{equation}
\sum_{m=\ell}^n
\left\{ 
p^{\leqslant n;(t)}(k+1,i;m,j) - p^{\leqslant n;(t)}(k,i;m,j) 
\right\} 
\ge 0,\qquad t \ge 0.
\label{eqn-18b}
\end{equation}
It also follows from (\ref{lim-P(t_n=infty)}) and
(\ref{defn-X_t^{<n}}) that, for any fixed $T>0$, the process
$\{(X_t^{\leqslant n},J_t^{\leqslant n}); \break 0 \le t < T\}$ converges to
the process $\{(X_t,J_t);0 \le t < T\}$ with probability one (w.p.1)
as $n \to \infty$ and thus, for each $(k,i;m,j) \in \bbF^2$,
\begin{eqnarray}
\lim_{n\to\infty}p^{\leqslant n;(t)}(k,i;m,j)
= p(k,i;m,j) \quad  \mbox{for all $t \in [0,T)$}.
\label{lim-p^{<n;(t)}(k,i;m,j)}
\end{eqnarray}
Applying the dominated convergence theorem to (\ref{eqn-18b}) and
using (\ref{lim-p^{<n;(t)}(k,i;m,j)}) yield
\[
\sum_{m=\ell}^{\infty}
\left\{ 
p^{(t)}(k+1,i;m,j) - p^{(t)}(k,i;m,j) 
\right\} 
\ge 0,\quad 0 \le t \le T,\ (k,i;\ell,j) \in \bbF^2,
\]
where $T > 0$ is arbitrarily fixed. Note here that
\[
\sum_{m=\ell}^{\infty} p^{(t)}(0,i;m,j) \ge 0, 
\quad t \ge 0,\ i \in \bbD,\ (\ell,j) \in \bbF.
\]
As a result, we obtain $\vc{T}_d^{-1} \vc{P}^{(t)} \vc{T}_d
\ge \vc{O}$ for all $t \ge 0$.

Next, we prove that statement (b) implies statement (a). To this end,
we consider the limit $\lim_{t \downarrow 0}
\vc{T}_d^{-1}(\vc{P}^{(t)} - \vc{I})\vc{T}_d / t$, that is,
\begin{eqnarray*}
&&
\dm\lim_{t \downarrow 0}\sum_{m=\ell}^{\infty} 
{ \dm\{p^{(t)}(0,i;m,j) - \chi_{(0,i)}(m,j)\} \over t}, 
\qquad (0,i;\ell,j) \in \bbF^2,
\end{eqnarray*}
and
\begin{eqnarray*}
&&\dm\lim_{t \downarrow 0}\sum_{m=\ell}^{\infty}
\Bigg[
{ p^{(t)}(k,i;m,j) - \chi_{(k,i)}(m,j) \over t}
\nonumber
\\
&& {} \qquad\quad 
- { p^{(t)}(k-1,i;m,j) - \chi_{(k-1,i)}(m,j) \over t}
\Bigg],
\quad(k,i;\ell,j) \in \bbF^2 \mbox{~with~} k \in \bbN,
\end{eqnarray*}
where $\chi_{(k,i)}(\ell,j)$, $(k,i;\ell,j) \in \bbF^2$, is given by 
\[
\chi_{(k,i)}(\ell,j)
= \left\{
\begin{array}{ll}
1, & (k,i) = (\ell,j),
\\
0, & (k,i) \neq (\ell,j).
\end{array}
\right.
\]

For all $(k,i) \in \bbF$, we have
\begin{eqnarray*}
\lefteqn{
\sum_{(m,j) \in \bbF} | p^{(t)}(k,i;m,j) - \chi_{(k,i)}(m,j) | 
}
\quad &&
\nonumber
\\
&\le& 1 - p^{(t)}(k,i;k,i) 
+ \sum_{(m,j) \in \bbF \setminus \{(k,i)\}} p^{(t)}(k,i;m,j) 
\nonumber
\\ 
&\le& 2\{ 1 - p^{(t)}(k,i;k,i) \}
\le 2t\, |q(k,i;k,i)|,
\end{eqnarray*}
where the last inequality follows from \cite[Theorem~II.3.1]{Chun67}.
Therefore, using dominated convergence theorem and (\ref{defn-Q}), we
obtain
\[
\lim_{t \downarrow 0}{ \vc{T}_d^{-1}(\vc{P}^{(t)} - \vc{I})\vc{T}_d \over t}
=
\vc{T}_d^{-1} \vc{Q} \vc{T}_d.
\]
Note here that $\vc{T}_d^{-1} \vc{P}^{(t)} \vc{T}_d \ge \vc{O}$ (due
to statement (b)), which implies that all the off-diagonal elements of
$\vc{T}_d^{-1} \vc{Q} \vc{T}_d$ are nonnegative, i.e., $\vc{Q} \in
\sfBM_d$. Consequently, statement (a) holds. \qed

\medskip

We now make the following assumption, in addition to
Assumption~\ref{assumpt-basic-zero}.
\begin{assumpt}\label{assumpt-Q_1-Q_2}
Suppose that $\vc{Q} \prec_d \wt{\vc{Q}}$ and either $\vc{Q} \in
\sfBM_d$ or $\wt{\vc{Q}} \in \sfBM_d$.
\end{assumpt}
\begin{lem}\label{lem-Q-2}
Suppose that Assumption~\ref{assumpt-Q_1-Q_2} holds. It then holds
that
\[
\xi(\ell,j)
= \sum_{\ell\in\bbZ_+^{\leqslant N}} q(k,i;\ell,j)
= \sum_{\ell\in\bbZ_+^{\leqslant N}} \wt{q}(k,i;\ell,j),
\qquad k \in \bbZ_+^{\leqslant N},\ i,j \in \bbD,
\]
which is constant with respect to $k$. Furthermore, $\vc{\Xi} =
(\xi(i,j))_{i,j\in\bbD}$ is the common infinitesimal generator of the
Markov chains $\{J_t;t\ge0\}$ and $\{\wt{J}_t;t\ge0\}$.
\end{lem}

\begin{proof}
It follows from $\vc{Q}\vc{T}_d \le \wt{\vc{Q}}\vc{T}_d$ (see
Definition~\ref{defn-dominance}) that
\[
\sum_{\ell=0}^{\infty} q(k,i;\ell,j)
\le \sum_{\ell=0}^{\infty} \wt{q}(k,i;\ell,j),
\qquad k \in \bbZ_+,\ i,j \in \bbD.
\]
Using this inequality, $\vc{Q}\vc{e}=\vc{0}$ (due to
Assumption~\ref{assumpt-basic-zero}) and $\wt{\vc{Q}}\vc{e} \le \vc{0}$
(see \cite[Section~1.2, Proposition~2.6]{Ande91}), we have
\[
0 = \sum_{\ell=0}^{\infty} \sum_{j\in\bbD} q(k,i;\ell,j)
\le \sum_{\ell=0}^{\infty} \sum_{j\in\bbD} \wt{q}(k,i;\ell,j) \le 0,
\qquad k \in \bbZ_+,\ i \in \bbD,
\]
which leads to 
\begin{equation}
\sum_{\ell=0}^{\infty} q(k,i;\ell,j) =
\sum_{\ell=0}^{\infty} \wt{q}(k,i;\ell,j),
\qquad k \in \bbZ_+,\ i,j \in \bbD.
\label{eqn-sum-q(k,i;l,j)}
\end{equation}
Furthermore, it follows from Lemma~\ref{prop-xi(i,j)}~(a) and either
$\vc{Q} \in \sfBM_d$ or $\wt{\vc{Q}} \in \sfBM_d$ that, for each
$(i,j) \in \bbD$, either of $\sum_{\ell=0}^{\infty}q(k,i;\ell,j)$ and
$\sum_{\ell=0}^{\infty}\wt{q}(k,i;\ell,j)$ is constant with respect to
$k \in \bbZ_+$. As a result, both sides of (\ref{eqn-sum-q(k,i;l,j)})
are constant with respect to $k$.  The remaining statement is immediate from Lemma~\ref{prop-xi(i,j)}~(b).
\end{proof}
\begin{lem}\label{lem-Q-3}
Suppose that Assumption~\ref{assumpt-Q_1-Q_2} holds. Furthermore, if
$\wt{\vc{Q}}$ is regular, then
\begin{enumerate}
\item $\vc{Q}$ is regular; and
\item $\vc{P}^{(t)} \prec_d \wt{\vc{P}}^{(t)}$ for all
  $t \ge 0$.
\end{enumerate}
\end{lem}

\begin{rem}
Lemma~\ref{lem-Q-3}~(b) is proved by using Lemma~\ref{lem-Q-3}~(a),
and the latter is proved based on
Lemma~\ref{lem2-continuous-ordering}~(a) (where $N$ is assumed to be
finite). Lemma~\ref{lem2-continuous-ordering}~(a) is proved without
Lemma~\ref{lem-Q-3}~(a) or (b) whereas
Lemma~\ref{lem2-continuous-ordering}~(b) (where $N$ is possibly
infinite) is proved by Lemma~\ref{lem-Q-3}~(b). For details, see the
proof of Lemma~\ref{lem2-continuous-ordering} in
Appendix~\ref{sec-pathwise-ordering}.
\end{rem}

\noindent
{\it Proof of Lemma~\ref{lem-Q-3}.~}We first provide some preliminaries to
the proof of statement
(a). Recall here that $\{(X_t^{\leqslant n},J_t^{\leqslant
  n});t\ge0\}$ is derived from $\{(X_t,J_t);t\ge0\}$, as shown in
(\ref{defn-X_t^{<n}}). Similarly, we define
$\{(\wt{X}_t^{\leqslant n},\wt{J}_t^{\leqslant
  n});t\ge0\}$ as a Markov chain with state space $\bbF^{\leqslant n}$
such that
\begin{eqnarray}
\wt{X}_t^{\leqslant n}
&=& 
\left\{
\begin{array}{ll}
\wt{X}_t,& 0 \le t < \wt{t}_n,
\\
n, & t \ge \wt{t}_n,
\end{array}
\right.
\quad 
\wt{J}_t^{\leqslant n} = \wt{J}_t, \quad t \ge 0,
\label{defn-X_{2,t}^{<n}}
\end{eqnarray}
where $\wt{t}_n = \inf\{t\ge0:\wt{X}_t \ge n\}$. Let $\wt{\vc{P}}^{\leqslant
  n;(t)}=(\wt{p}^{\leqslant
  n;(t)}(k,i;\ell,j))_{(k,i),(\ell,j)\in\bbF}$, $t \ge 0$, and
$\wt{\vc{Q}}^{\leqslant n}=(\wt{q}^{\leqslant
  n}(k,i;\ell,j))_{(k,i),(\ell,j)\in\bbF}$ denote the transition matrix function
and infinitesimal generator, respectively, of the Markov chain
$\{(\wt{X}_t^{\leqslant n},\wt{J}_t^{\leqslant n})\}$,
i.e.,
\begin{eqnarray}
\wt{p}^{\leqslant n;(t)}(k,i;\ell,j)
&=& 
\PP( \wt{X}_t^{\leqslant n} = \ell, \wt{J}_t^{\leqslant n} = j
\mid \wt{X}_0^{\leqslant n} = k, \wt{J}_0^{\leqslant n} = i),
\label{defn-wt{p}^{<=n;(t)}}
\\
\wt{q}^{\leqslant n}(k,i;\ell,j)
&=& \lim_{t \downarrow 0}
{ \wt{p}^{\leqslant n;(t)}(k,i;\ell,j) - \chi_{(k,i)}(\ell,j) \over t }.
\nonumber
\end{eqnarray}
It then follows from (\ref{defn-X_{2,t}^{<n}}) and Lemma~\ref{lem-Q-2} 
that, for $i,j \in \bbD$,
\begin{equation}
\wt{q}^{\leqslant n}(k,i;\ell,j)
=\left\{
\begin{array}{ll}
\wt{q}(k,i;\ell,j), & k,\ell \in \bbZ_+^{\leqslant n-1},
\\
\dm\sum_{m=n}^{\infty}\wt{q}(k,i;m,j), & k\in \bbZ_+^{\leqslant n-1},~\ell = n,
\\
\dm\sum_{m=0}^{\infty}\wt{q}(n,i;m,j) = \xi(i,j), & k=\ell=n,
\\
0, & \mbox{otherwise}.
\end{array}
\right.
\label{defn-widetilde{q}^{<n}(k,i;l,j)}
\end{equation}
Using (\ref{defn-q^{<n}(k,i;l,j)}), (\ref{defn-widetilde{q}^{<n}(k,i;l,j)})
 and $\vc{Q}\vc{T}_d \le \wt{\vc{Q}} \vc{T}_d$, we have
\begin{equation}
\vc{Q}^{\leqslant n}\vc{T}_d^{\leqslant n} \le
\wt{\vc{Q}}^{\leqslant n}\vc{T}_d^{\leqslant n}.
\label{ineqn-Q_1^{<n}-Q_2^{<n}}
\end{equation}
Note here that $\vc{Q} \in \sfBM_d$ (resp.\ $\wt{\vc{Q}} \in \sfBM_d$)
implies $\vc{Q}^{\leqslant n} \in \sfBM_d$
(resp.\ $\wt{\vc{Q}}^{\leqslant n} \in \sfBM_d$). Therefore, according
to Lemma~\ref{lem2-continuous-ordering}~(a), we assume, without loss
of generality, that
\begin{equation}
X_t^{\leqslant n} \le \wt{X}_t^{\leqslant n},\quad 
J_t^{\leqslant n} = \wt{J}_t^{\leqslant n} \quad \mbox{for all $t > 0$},
\label{pathwise-ordering-X_t^{<=n}}
\end{equation}
given that $X_0^{\leqslant n} \le \wt{X}_0^{\leqslant n}$ and
$J_0^{\leqslant n} = \wt{J}_0^{\leqslant n}$. 

We now prove statement
(a) by contradiction. To this end, we suppose that $\vc{Q}$ is not regular. Thus, there exist some
$T_{\infty} \in (0,\infty)$ and $(k_0,i_0) \in \bbF$ such that
\begin{equation}
\PP(\cap_{n>k_0}\{t_n \le T_{\infty}\} 
\mid X_0 = k_0, J_0 = i_0) > 0.
\label{prob-t_n<=T_{infty}}
\end{equation}
In addition, it follows from (\ref{defn-X_t^{<n}}),
(\ref{defn-X_{2,t}^{<n}}) and (\ref{pathwise-ordering-X_t^{<=n}}) that
if $X_0 = \wt{X}_0 = k_0$ and $J_0 = \wt{J}_0 = i_0$ then
\[
X_0^{\leqslant n} = \wt{X}_0^{\leqslant n} = k_0, \qquad
J_0^{\leqslant n} = \wt{J}_0^{\leqslant n} = i_0,\qquad n > k_0,
\]
and
\begin{eqnarray*}
\cap_{n>k_0}\{t_n \le T_{\infty}\} 
&\Rightarrow& \cap_{n>k_0}\{X_t^{\leqslant n} \ge n~\mbox{for all}~ t \ge T_{\infty}\}
\nonumber
\\
&\Rightarrow& \cap_{n>k_0}\{\wt{X}_t^{\leqslant n} \ge n~\mbox{for all}~  t \ge T_{\infty}\}
\nonumber
\\
&\Rightarrow& \cap_{n>k_0}\{\wt{t}_n \le T_{\infty}\}.
\end{eqnarray*}
Combining these and (\ref{prob-t_n<=T_{infty}}) yields
\begin{eqnarray*}
\PP(\cap_{n>k_0}\{\wt{t}_n \le T_{\infty}\} 
\mid \wt{X}_0 = k_0, \wt{J}_0 = i_0) > 0,
\end{eqnarray*}
which is inconsistent with the assumption that $\wt{\vc{Q}}$ is
regular. As a result, $\vc{Q}$ must be regular.

Next, we prove statement (b). According to statement (a), the two
infinitesimal generators $\vc{Q}$ and $\wt{\vc{Q}}$ are regular and
thus
\begin{eqnarray}
\vc{Q}\vc{e} &=& \wt{\vc{Q}}\vc{e}=\vc{0},
\label{eqn-Qe-wt{Q}e=0}
\\
\PP(\lim_{n\to\infty}t_n = \infty) 
&=& \PP(\lim_{n\to\infty}\wt{t}_n = \infty) = 1.
\label{lim-t_n-wt{t}_n}
\end{eqnarray}
It follows from (\ref{defn-q^{<n}(k,i;l,j)}), (\ref{defn-widetilde{q}^{<n}(k,i;l,j)}) and (\ref{eqn-Qe-wt{Q}e=0}) that $\vc{I} +
\varsigma_n^{-1}\vc{Q}^{\leqslant n}$ and $\vc{I}
+\varsigma_n^{-1}\wt{\vc{Q}}^{\leqslant n}$ are stochastic, where
\[
\varsigma_n = \max_{(k,i)\in\bbF^{\leqslant n}} 
\max\left( |q^{\leqslant n}(k,i;k,i)|, |\wt{q}^{\leqslant n}(k,i;k,i)| \right)
< \infty.
\]
It also follows from (\ref{ineqn-Q_1^{<n}-Q_2^{<n}}) and
\cite[Proposition 2.3~(b)]{Masu15-ADV} that
\[
(\vc{I} + \varsigma_n^{-1}\vc{Q}^{\leqslant n})^m\vc{T}_d^{\leqslant n} 
\le (\vc{I} + \varsigma_n^{-1}\wt{\vc{Q}}^{\leqslant n})^m\vc{T}_d^{\leqslant n},
\qquad m \in \bbZ_+,
\]
and thus, for $t \ge 0$,
\begin{eqnarray}
\exp\{\vc{Q}^{\leqslant n} t\} \vc{T}_d^{\leqslant n} 
&=& \sum_{m=0}^{\infty} \rme^{-\varsigma t} {(\varsigma t)^m \over m!}
(\vc{I} + \varsigma_n^{-1}\vc{Q}^{\leqslant n})^m\vc{T}_d^{\leqslant n} 
\nonumber
\\
&\le& 
\sum_{m=0}^{\infty} \rme^{-\varsigma t} {(\varsigma t)^m \over m!}
(\vc{I} + \varsigma_n^{-1}\wt{\vc{Q}}^{\leqslant n})^m\vc{T}_d^{\leqslant n}
\nonumber
\\
&=&
\exp\{\wt{\vc{Q}}^{\leqslant n} t\} \vc{T}_d^{\leqslant n}.
\label{ineqn-exp{Q_1^{<n}}-exp{Q_2^{<n}}}
\end{eqnarray}
By definition (see (\ref{defn-p^{<=n;(t)}}) and (\ref{defn-wt{p}^{<=n;(t)}})), $p^{\leqslant n;(t)}(k,i;\ell,j)$ and $\wt{p}^{\leqslant
  n;(t)}(k,i;\ell,j)$ are equal to the $(k,i;\ell,j)$th elements of
$\exp\{\vc{Q}^{\leqslant n} t\}$ and $\exp\{\wt{\vc{Q}}^{\leqslant n}
t\}$, respectively. Therefore, from
(\ref{ineqn-exp{Q_1^{<n}}-exp{Q_2^{<n}}}), we have, for $t \ge 0$ and
$(k,i;\ell,j) \in \bbF^{\leqslant n} \times \bbF^{\leqslant n}$,
\begin{equation}
\sum_{m=\ell}^n 
\left\{
\wt{p}^{\leqslant n;(t)}(k,i;m,j) - p^{\leqslant n;(t)}(k,i;m,j) 
\right\} \ge 0.
\label{eqn-26}
\end{equation}
Furthermore, combining (\ref{defn-X_t^{<n}}), (\ref{defn-X_{2,t}^{<n}}) and (\ref{lim-t_n-wt{t}_n}), we obtain, for any fixed $T > 0$,
\begin{eqnarray}
\lim_{n\to\infty}p^{\leqslant n;(t)}(k,i;\ell,j) 
&=& p^{(t)}(k,i;\ell,j),
~~~~ t \in [0,T),\ (k,i;\ell,j) \in \bbF^2,\quad
\label{add-eqn-08-1}
\\
\lim_{n\to\infty}\wt{p}^{\leqslant n;(t)}(k,i;\ell,j) 
&=& \wt{p}^{(t)}(k,i;\ell,j),
~~~~ t \in [0,T),\ (k,i;\ell,j) \in \bbF^2.\quad
\label{add-eqn-08-2}
\end{eqnarray}
Applying (\ref{add-eqn-08-1}), (\ref{add-eqn-08-2}) and the dominated
convergence theorem to (\ref{eqn-26}) yields
\[
\sum_{m=\ell}^{\infty} 
\left\{\wt{p}^{(t)}(k,i;m,j) - p^{(t)}(k,i;m,j) \right\} \ge 0,
\qquad t \in [0,T),\ (k,i;\ell,j) \in \bbF^2.
\]
Letting $T \to\infty$ in the above inequality, we have $\vc{P}^{(t)}
\vc{T}_d \le \wt{\vc{P}}^{(t)} \vc{T}_d$ for all $t \ge 0$. \qed

\begin{lem}\label{lem-Q-4}
Suppose that Assumption~\ref{assumpt-Q_1-Q_2} holds. Furthermore,
suppose that $\wt{\vc{Q}}$ is regular and irreducible. Under these
conditions, the following are true:
\begin{enumerate}
\item If $\wt{\vc{Q}}$ is recurrent, then $\vc{Q}$ has exactly
  one recurrent communicating class $\bbC \subseteq \bbF^{\leqslant
    N}$ that includes the states $\{(0,i);i\in\bbD\}$, which is
  reachable from all the other states w.p.1.
\item Furthermore, if $\wt{\vc{Q}}$ is positive recurrent, then the
  unique communicating class $\bbC$ is positive recurrent and
  $\vc{\pi} \prec_d \wt{\vc{\pi}}$, where
  $\vc{\pi}:=(\pi(k,i))_{(k,i)\in\bbF^{\leqslant N}}$ and
  $\wt{\vc{\pi}}:=(\wt{\pi}(k,i))_{(k,i)\in\bbF^{\leqslant N}}$ are
  the unique stationary distribution vectors of $\vc{Q}$ and
  $\wt{\vc{Q}}$, respectively.
\end{enumerate}
\end{lem}

\begin{rem}
An irreducible infinitesimal generator of a finite order is ergodic
\cite[Theorems 3.3 and 5.2 and Definitions 5.1 and 5.2]{Brem99} and
thus is regular (see Remark~\ref{rem-regular}). Therefore, if
Assumption~\ref{assumpt-Q_1-Q_2} holds for $N < \infty$ and
$\wt{\vc{Q}}$ is irreducible, then statement (b) of
Lemma~\ref{lem-Q-4} is true.
\end{rem}

\begin{proofof}{Lemma~\ref{lem-Q-4}}
We first prove statement (a). To this end, we assume, without loss of
generality, that the two Markov chains $\{(X_t,J_t);t \ge 0\}$ and
$\{(\wt{X},\wt{J}_t);t \ge 0\}$ are pathwise ordered as follows (see
Lemma~\ref{lem2-continuous-ordering}):
\begin{equation}
X_t \le \wt{X}_t,
\quad
J_t = \wt{J}_t\quad \mbox{for all $t \ge 0$}.
\label{eqn-pathwise-ordering}
\end{equation}
It follows from (\ref{eqn-pathwise-ordering}), together with the
irreducibility and recurrence of $\{(\wt{X},\wt{J}_t)\}$, that
$\{(X_t,J_t)\}$ can reach any state in $\{(0,i);i\in\bbD\}$ from all
the states in the state space $\bbF$ w.p.1. Therefore, $\vc{Q}$ has
exactly one recurrent communicating class $\bbC \subseteq \bbF$ such
that $\bbC \supseteq \{(0,i);i\in\bbD\}$.

Next, we prove statement (b). For this purpose, we additionally assume
that $\wt{\vc{Q}}$ and thus $\{(\wt{X},\wt{J}_t)\}$ are positive
recurrent (i.e., ergodic), under which $\wt{\vc{Q}}$ has the unique
stationary distribution vector $\wt{\vc{\pi}}$ (see
Remark~\ref{rem-pi}) and
\begin{equation}
\lim_{t\to\infty}\wt{\vc{P}}^{(t)}
= \vc{e}\wt{\vc{\pi}}.
\label{lim-P_2^{(t)}}
\end{equation}
Furthermore, the ergodicity of $\wt{\vc{Q}}$ and
(\ref{eqn-pathwise-ordering}) imply that the mean first passage time
of $\{(X_t,J_t)\}$ to each state in $\{(0,i);i\in\bbD\}$ is finite for
any given initial state, which leads to the result that the unique
communicating class $\bbC$ of $\vc{Q}$ is positive recurrent. Therefore, it
follows from \cite[Theorems II.10.1 and II.10.2]{Chun67} and
\cite[Section 5.4, Theorem 4.5]{Ande91} that $\vc{Q}$ has the unique
stationary distribution vector $\vc{\pi}$ and
\begin{equation}
\lim_{t\to\infty}\vc{P}^{(t)}
= \vc{e}\vc{\pi}.
\label{lim-P_1^{(t)}}
\end{equation}
It also follows from Lemma~\ref{lem-Q-3}~(b) that
$\vc{P}^{(t)}\vc{T}_d \le \wt{\vc{P}}^{(t)}\vc{T}_d$ for $t \ge 0$.
From this inequality together with (\ref{lim-P_2^{(t)}}),
(\ref{lim-P_1^{(t)}}) and the dominated convergence theorem, we obtain
$\vc{e}\vc{\pi}\vc{T}_d \le \vc{e}\wt{\vc{\pi}}\vc{T}_d$ and thus
$\vc{\pi}\vc{T}_d \le \wt{\vc{\pi}}\vc{T}_d$.
\end{proofof}

\section{Block-augmented truncations}\label{sec-LBC}

In this section, we discuss the block-augmented truncation of
infinite-order block-structured infinitesimal generators. Thus, we
assume that Assumption~\ref{assumpt-basic-zero} holds for $N =
\infty$, i.e., $\vc{Q}$ is an $|\bbF| \times |\bbF|$ stable and
conservative infinitesimal generator.

We begin with the definition of the block-augmented truncation of
$\vc{Q}$.
\begin{defn}\label{defn-(n)Q_*}
Let
$\presub{n}\vc{Q}_{\ast}=(\presub{n}q_{\ast}(k,i;\ell,j))_{(k,i),(\ell,j)\in\bbF}$
denote an infinitesimal generator such that, for $i,j\in\bbD$,
\begin{align*}
&&
\presub{n}q_{\ast}(k,i;\ell,j) &\ge q(k,i;\ell,j), 
& & k \in \bbZ_+,\ 0 \le \ell \le n, &&
\\
&&
\presub{n}q_{\ast}(k,i;k,j) &= q(k,i;k,j), 
& & k = \ell \ge n+1, &&
\\
&&
\presub{n}q_{\ast}(k,i;\ell,j) &= 0, 
& & k \in \bbZ_+,\ \ell \ge n+1,\ \ell \neq k, &&
\\
&&
\dm\sum_{\ell=0}^{\infty}\presub{n}q_{\ast}(k,i;\ell,j)
 &= \sum_{\ell=0}^{\infty}q(k,i;\ell,j),
&& k \in \bbZ_+. &&
\end{align*}
The infinitesimal generator $\presub{n}\vc{Q}_{\ast}$ is called
a {\it
  block-augmented northwest-corner truncation (block-augmented truncation, for short)} of $\vc{Q}$. 
\end{defn}

Clearly, $\presub{n}\vc{Q}_{\ast}$ has the following form:
\begin{equation}
\presub{n}\vc{Q}_{\ast}
= 
\left(
\begin{array}{c|ccccc}
\presub{n}\vc{Q}_{\ast}^{\leqslant n} 
& \vc{O}
& \vc{O}
& \vc{O}
& \vc{O}
& \cdots
\\
\hline
\ast
& \ast
& \vc{O}
& \vc{O}
& \vc{O}
& \cdots
\\
  \ast
& \vc{O}
& \ast
& \vc{O}
& \vc{O}
& \cdots
\\
  \ast
& \vc{O}
& \vc{O}
& \ast
& \vc{O}
& \cdots
\\
  \ast
& \vc{O}
& \vc{O}
& \vc{O}
& \ast
& \cdots
\\
\vdots
& \vdots
& \vdots
& \vdots
& \vdots
& \ddots
\end{array}
\right),
\label{structure-(n)Q_*}
\end{equation}
where $\presub{n}\vc{Q}_{\ast}^{\leqslant n}$ denotes the
$|\bbF^{\leqslant n}| \times |\bbF^{\leqslant n}|$ northwest-corner of
$\presub{n}\vc{Q}_{\ast}$.  It may seem more reasonable to define
$\presub{n}\vc{Q}_{\ast}^{\leqslant n}$ as a block-augmented
truncation of $\vc{Q}$, instead of
$\presub{n}\vc{Q}_{\ast}$. Nevertheless, we adopt
$\presub{n}\vc{Q}_{\ast}$ in order to perform algebraic operations on
the original infinitesimal generator and its block-augmented
truncation.

For further discussion, we assume that $\vc{Q}$ is irreducible, under
which we present some fundamental results on the block-augmented
truncation.
\begin{lem}\label{lem-(n)Q_*-communicating}
If $\vc{Q}$ is irreducible, then $\presub{n}\vc{Q}_{\ast}$ has no
closed communicating classes in $\bbF^{>n}:=\bbF \setminus
\bbF^{\leqslant n}$.
\end{lem}

\begin{proof}
We assume that there exists a closed communicating class $\bbC$ in
$\bbF^{>n}$.  Since $\presub{n}\vc{Q}_{\ast}$ is block-diagonal in
$\bbF^{>n}$, the closed communicating class $\bbC$ must be within a
set $\{(k,i);i\in \bbD\}$ for some $k \ge n+1$, which implies that the
{\it principal submatrix}
$(\presub{n}q_{\ast}(k,i;\ell,j))_{(k,i),(\ell,j)\in\bbC}$ of
$\presub{n}\vc{Q}_{\ast}$ is a conservative infinitesimal
generator. From this result and Definition~\ref{defn-(n)Q_*} of
$\presub{n}\vc{Q}_{\ast}$, we have
\begin{equation}
\sum_{(\ell,j)\in\bbC}q(k,i;\ell,j) 
= \sum_{(\ell,j)\in\bbC}\presub{n}q_{\ast}(k,i;\ell,j) = 0,
\qquad(k,i) \in \bbC.
\label{sum-q(k,i;l,j)-in-C}
\end{equation}
Note here that the {\it whole} matrices $\vc{Q}$ and
$\presub{n}\vc{Q}_{\ast}$ are also conservative infinitesimal
generators, i.e., for $(k,i) \in \bbF$,
\begin{eqnarray}
q(k,i;\ell,j) &\ge& 0, \quad \presub{n}q_{\ast}(k,i;\ell,j) \ge 0,
\quad (\ell,j) \in \bbF \setminus \{(k,i)\},
\\
\sum_{(\ell,j)\in\bbF}q(k,i;\ell,j) 
&=& \sum_{(\ell,j)\in\bbF}\presub{n}q_{\ast}(k,i;\ell,j) = 0.
\label{sum-q(k,i;l,j)-in-F}
\end{eqnarray}
It follows from
(\ref{sum-q(k,i;l,j)-in-C})--(\ref{sum-q(k,i;l,j)-in-F}) that
\[
q(k,i;\ell,j) = \presub{n}q_{\ast}(k,i;\ell,j) = 0 \quad \mbox{for
  all $(k,i) \in \bbC$ and $(\ell,j) \in \bbF \setminus \bbC$}.
\]
Therefore, the original Markov chain $\{(X(t),J(t))\}$ with
infinitesimal generator $\vc{Q}$ cannot move out of $\bbC \subset
\bbF^{>n}$. This contradicts to the irreducibility of
$\{(X(t),J(t))\}$. As a result, $\presub{n}\vc{Q}_{\ast}$ has no
closed communicating classes in $\bbF^{>n}$.
\end{proof}

Lemma~\ref{lem-(n)Q_*-communicating} shows that any closed
communicating class of $\presub{n}\vc{Q}_{\ast}$ is finite (because it
is in the finite set $\bbF^{\leqslant n}$) and thus is positive
recurrent due to the combination of \cite[Section~5.1,
  Proposition~1.4]{Ande91} and \cite[Theorem~4.8]{Kulk10}. Therefore,
it follows from \cite[Section~5.4, Theorem~4.5]{Ande91} that
$\presub{n}\vc{Q}_{\ast}$ has at least one stationary distribution
vector.

We now have the following result.
\begin{lem}\label{lem-(n)pi_*(k,i)}
Suppose that $\vc{Q}$ is irreducible. Let
$\presub{n}\vc{\pi}_{\ast}:=(\presub{n}\pi_{\ast}(k,i))_{(k,i)\in\bbF}$
denote an arbitrary stationary distribution vector of
$\presub{n}\vc{Q}_{\ast}$. It then holds that
\begin{equation}
\presub{n}\pi_{\ast}(k,i) = 0
\quad \mbox{for all $(k,i) \in \bbF^{>n}$}.
\label{eqn-(n)pi_*(k,i)=0}
\end{equation}
\end{lem}

\begin{proof}
By definition,
\[
\presub{n}\vc{\pi}_{\ast}\presub{n}\vc{Q}_{\ast} = \vc{0}.
\]
Thus, from (\ref{structure-(n)Q_*}) and Definition~\ref{defn-(n)Q_*},
we have
\begin{equation}
\presub{n}\vc{\pi}_{\ast}(k) \presub{n}\vc{Q}_{\ast}(k;k) = \vc{0}
\quad \mbox{for all $k \ge n+1$},
\label{eqn-(n)pi(k)}
\end{equation}
where
$\presub{n}\vc{\pi}_{\ast}(k)=(\presub{n}\pi_{\ast}(k,i))_{i\in\bbD}$
and $\presub{n}\vc{Q}_{\ast}(k;k) = (q(k,i;k,j))_{i,j\in\bbD}$. Note
here that $\presub{n}\vc{Q}_{\ast}(k;k)$, $k \ge n+1$, is the
infinitesimal generator ($q$-matrix) of a Markov chain restricted to
the set of states $\{(k,i);i \in \bbD\} \subset \bbF^{>n}$. Note also
that all the states in $\bbF^{>n}$ are transient, as shown in
Lemma~\ref{lem-(n)Q_*-communicating}. Therefore, for all $k \ge n+1$,
$\presub{n}\vc{Q}_{\ast}(k;k)$ is non-singular \cite[Section
  8.6.2]{Brem99}. Post-multiplying both sides of (\ref{eqn-(n)pi(k)})
by $\presub{n}\vc{Q}_{\ast}(k;k)^{-1}$, we obtain
$\presub{n}\vc{\pi}_{\ast}(k) = \vc{0}$ for all $n \ge k+1$, i.e.,
(\ref{eqn-(n)pi_*(k,i)=0}) holds.
\end{proof}

It follows from Assumption~\ref{assumpt-basic-zero} and
Definition~\ref{defn-(n)Q_*} that $\presub{n}\vc{Q}_{\ast}$ is stable
and conservative. In addition, the irreducibility of $\vc{Q}$ makes
$\presub{n}\vc{Q}_{\ast}$ regular, as stated in the following lemma.
\begin{lem}\label{lem-(n)Q_*-regular}
If $\vc{Q}$ is irreducible, then $\presub{n}\vc{Q}_{\ast}$ is regular.
\end{lem}

\begin{proof}
We assume that $\presub{n}\vc{Q}_{\ast}$ is not regular, i.e., the
equation
\begin{equation}
\presub{n}\vc{Q}_{\ast}\vc{x} = \gamma \vc{x},
\quad \vc{0} \le \vc{x} \le \vc{e}
\label{eqn-(n)Q_*x}
\end{equation}
has a nontrivial solution for some $\gamma > 0$ (see
(\ref{eqn-Q-regular})). Let
$\breve{\vc{x}}:=(\breve{x}(k,i))_{(k,i)\in\bbF} \ge \vc{0},\neq
\vc{0}$ denote such a solution. It then follows from
(\ref{structure-(n)Q_*}) and (\ref{eqn-(n)Q_*x}) that
\[
\left[ \presub{n}\vc{Q}_{\ast}(k,k) - \gamma \vc{I} \right]
\breve{\vc{x}}(k) = \vc{0},
\qquad k \ge n+1,
\]
where $\breve{\vc{x}}(k) = (\breve{x}(k,i))_{i\in\bbD}$ for $k \in
\bbZ_+$.  Since $\presub{n}\vc{Q}_{\ast}(k,k)$ is a $q$-matrix, the
matrix $\presub{n}\vc{Q}_{\ast}(k,k) - \gamma \vc{I}$ is nonsingular
and thus $\breve{\vc{x}}(k) = \vc{0}$ for all $k \ge n+1$. Therefore,
$\breve{x}(k',i') > 0$ for some $(k',i') \in \bbF^{\leqslant n}$ due
to $\breve{\vc{x}} \neq \vc{0}$.

Recall that $\bbF^{\leqslant n}$ is a closed set of the states of
$\presub{n}\vc{Q}_{\ast}$ (see (\ref{structure-(n)Q_*})). Thus, there
exists a closed communicating class including the state $(k',i') \in
\bbF^{\leqslant n}$, which implies that there exists a stationary
distribution vector
$\presub{n}\breve{\vc{\pi}}_{\ast}:=(\presub{n}\breve{\pi}_{\ast}(k,i))_{(k,i)\in\bbF}$
of $\presub{n}\vc{Q}_{\ast}$ such that
$\presub{n}\breve{\pi}_{\ast}(k',i') > 0$ \cite[Section 5.4, Theorem
  4.5]{Ande91}. Therefore, $\presub{n}\breve{\vc{\pi}}_{\ast}
\breve{\vc{x}} \ge \presub{n}\breve{\pi}_{\ast}(k',i')\breve{x}(k',i')
> 0$. On the other hand, pre-multiplying both sides of
(\ref{eqn-(n)Q_*x}) by $\presub{n}\breve{\vc{\pi}}_{\ast}$, we have $0
= \gamma \cdot \presub{n}\breve{\vc{\pi}}_{\ast}\breve{\vc{x}} > 0$,
which is a contradiction. As a result, the assumption at the beginning
is denied, i.e., $\presub{n}\vc{Q}_{\ast}$ is regular.
\end{proof}

We consider two special cases of the block-augmented truncation.  Let
$\presub{n}\vc{Q}_n=(\presub{n}q_n(k,i;\ell,j))_{(k,i),(\ell,j)\in\bbF}$
denote an infinitesimal generator such that, for $i,j\in\bbD$,
\begin{eqnarray}
\lefteqn{
\presub{n}q_n(k,i;\ell,j)
}
\quad &&
\nonumber
\\
&=& 
\left\{
\begin{array}{ll}
q(k,i;\ell,j), & k \in \bbZ_+,\ 0 \le \ell \le n-1,
\\
q(k,i;n,j)
+ \dm\sum_{m>n,\, m \neq k} q(k,i;m,j), & k \in \bbZ_+,\ \ell = n,
\\
q(k,i;k,j), & k = \ell \ge n+1,
\\
0, & \mbox{otherwise}.
\end{array}
\right.
\label{def-trunc{q}_n(k,i;l,j)}
\end{eqnarray}
 Let
 $\presub{n}\vc{Q}_0=(\presub{n}q_0(k,i;\ell,j))_{(k,i),(\ell,j)\in\bbF}$
 denote an infinitesimal generator such that, for $i,j\in\bbD$,
\begin{eqnarray}
\lefteqn{
\presub{n}q_0(k,i;\ell,j)
}
\quad &&
\nonumber
\\
&=& 
\left\{
\begin{array}{ll}
q(k,i;0,j) + \dm\sum_{m>n,\, m \neq k} q(k,i;m,j), 
& k \in \bbZ_+,\ \ell=0,
\\
q(k,i;\ell,j), & k \in \bbZ_+,\ 1 \le \ell \le n,
\\
q(k,i;k,j), & k = \ell \ge n+1,
\\
0, & \mbox{otherwise}.
\end{array}
\right.
\label{def-trunc{q}_0(k,i;l,j)}
\end{eqnarray}
We refer to $\presub{n}\vc{Q}_n$ as the {\it last-column-block-augmented
  northwest-corner truncation (LC-block-augmented truncation, for short)}
of $\vc{Q}$. We also refer to $\presub{n}\vc{Q}_0$ as the {\it
  first-column-block-augmented northwest-corner truncation (FC-block-augmented truncation, for short)} of $\vc{Q}$.

Let $\presub{n}\vc{\pi}_n:=(\presub{n}\pi_n(k,i))_{(k,i)\in\bbF}$ and
$\presub{n}\vc{\pi}_0:=(\presub{n}\pi_0(k,i))_{(k,i)\in\bbF}$ denote
the stationary distribution vectors of $\presub{n}\vc{Q}_n$ and
$\presub{n}\vc{Q}_0$, respectively. It then follows from
Lemma~\ref{lem-(n)pi_*(k,i)} that
\[
\presub{n}\pi_n(k,i) = 
\presub{n}\pi_0(k,i) = 0
\quad \mbox{for all $(k,i) \in \bbF^{>n}$}.
\]
The following theorem is a generalization of \cite[Theorem~3.6]{LiHai00}.
\begin{thm}\label{thm-block-augmentation}
If $\vc{Q}$ is ergodic (i.e., irreducible and positive recurrent) and
$\vc{Q} \in \sfBM_d$, then the following are true:
\begin{enumerate}
\item An arbitrary block-augmented truncation
  $\presub{n}\vc{Q}_{\ast}$ has the unique stationary distribution
  vector $\presub{n}\vc{\pi}_{\ast}$ and
\begin{equation}
\presub{n}\vc{\pi}_0 \prec_d 
\presub{n}\vc{\pi}_{\ast} \prec_d 
\presub{n}\vc{\pi}_n \prec_d 
\vc{\pi},\qquad n \in \bbN.
\label{order-pi's}
\end{equation}
\item As $n \to \infty$, $\{\presub{n}\vc{\pi}_{\ast}\}$ converges to
  $\vc{\pi}$ elementwise, i.e.,
\begin{equation}
\lim_{n\to\infty}\presub{n}\vc{\pi}_{\ast} = \vc{\pi}.
\label{lim-(n)pi_n}
\end{equation}
\end{enumerate}
\end{thm}

\begin{proof}
We first prove statement (a).  Definition~\ref{defn-(n)Q_*},
(\ref{def-trunc{q}_n(k,i;l,j)}) and (\ref{def-trunc{q}_0(k,i;l,j)})
imply that
\begin{equation*}
\presub{n}\vc{Q}_0 \prec_d 
\presub{n}\vc{Q}_{\ast} \prec_d 
\presub{n}\vc{Q}_n \prec_d \vc{Q},
\qquad n \in \bbN,
\end{equation*}
which leads to 
\[
\presub{n}\vc{Q}_0 \prec_d 
\presub{n}\vc{Q}_n \prec_d 
\vc{Q},
\quad 
\presub{n}\vc{Q}_{\ast} \prec_d 
\presub{n}\vc{Q}_n \prec_d 
\vc{Q},\quad n \in \bbN.
\]
Note here that $\vc{Q} \in \sfBM_d$ implies $\presub{n}\vc{Q}_n \in
\sfBM_d$ for all $n \in \bbN$.  It follows from these results and
Lemma~\ref{lem-Q-4}~(b) that $\presub{n}\vc{\pi}_0$,
$\presub{n}\vc{\pi}_{\ast}$ and $\presub{n}\vc{\pi}_n$ are the unique
stationary distribution vectors of $\presub{n}\vc{Q}_0$,
$\presub{n}\vc{Q}_{\ast}$ and $\presub{n}\vc{Q}_n$, respectively, and
\begin{equation*}
\presub{n}\vc{\pi}_0 \prec_d 
\presub{n}\vc{\pi}_n \prec_d 
\vc{\pi},
\quad 
\presub{n}\vc{\pi}_{\ast} \prec_d 
\presub{n}\vc{\pi}_n \prec_d 
\vc{\pi},\quad n \in \bbN.
\end{equation*}
Therefore, it remains to prove that $\presub{n}\vc{\pi}_0 \prec_d
\presub{n}\vc{\pi}_{\ast}$ for $n \in \bbN$. To this end, we define
$\presub{n}\vc{Q}_0^{\leqslant n}$ and
$\presub{n}\vc{Q}_{\ast}^{\leqslant n}$ as the $|\bbF^{\leqslant n}|
\times |\bbF^{\leqslant n}|$ northwest-corner truncations of
$\presub{n}\vc{Q}_0$ and $\presub{n}\vc{Q}_{\ast}$, respectively. We
also define $\presub{n}\vc{\pi}_0^{\leqslant n} =
(\presub{n}\pi_0(k,i))_{(k,i)\in\bbF^{\leqslant n}}$ and
$\presub{n}\vc{\pi}_{\ast}^{\leqslant n} =
(\presub{n}\pi_{\ast}(k,i))_{(k,i)\in\bbF^{\leqslant n}}$,
respectively. Lemma~\ref{lem-(n)pi_*(k,i)} implies that
$\presub{n}\vc{\pi}_0^{\leqslant n}$ and
$\presub{n}\vc{\pi}_{\ast}^{\leqslant n}$ are the unique stationary
distribution vectors of $\presub{n}\vc{Q}_0^{\leqslant n}$ and
$\presub{n}\vc{Q}_{\ast}^{\leqslant n}$, respectively.  Note here that
$\presub{n}\vc{Q}_0^{\leqslant n} \vc{T}_d^{\leqslant n} \le
\presub{n}\vc{Q}_{\ast}^{\leqslant n}\vc{T}_d^{\leqslant n}$ for $n
\in \bbN$. Thus, proceeding as in the derivation of
(\ref{ineqn-exp{Q_1^{<n}}-exp{Q_2^{<n}}}), we can readily show that,
for $n \in \bbN$,
\[
\exp\{ \presub{n}\vc{Q}_0^{\leqslant n} t\} \vc{T}_d^{\leqslant n}
\exp\{ \presub{n}\vc{Q}_{\ast}^{\leqslant n} t\}\vc{T}_d^{\leqslant n},
\qquad t \ge 0.
\]
Letting $t \to \infty$ in the above inequality, we have
\[
\vc{e} \cdot \presub{n}\vc{\pi}_0^{\leqslant n}\vc{T}_d^{\leqslant n} \le
\vc{e} \cdot \presub{n}\vc{\pi}_{\ast}^{\leqslant n}\vc{T}_d^{\leqslant n},
\qquad n \in \bbN,
\]
which shows that $\presub{n}\vc{\pi}_0 \prec_d
\presub{n}\vc{\pi}_{\ast}$ for $n \in \bbN$. Consequently, statement
(a) has been proved.

Next, we prove statement (b). It follows from statement (a) that, for
$n \in \bbN$, $\presub{n}\vc{\pi}_{\ast} \prec_d \vc{\pi}$, that is,
\[
\sum_{k=\ell}^{\infty}\presub{n}\pi_{\ast}(k,i) \le
\sum_{k=\ell}^{\infty}\pi(k,i),\qquad (\ell,j) \in \bbF.
\]
Therefore, for any $\varepsilon \in (0,1)$, there
exists some $k_{\varepsilon} \in \bbZ_+$ such that
\begin{equation*}
\sum_{(k,i) \in \bbF^{>k_{\varepsilon}}} \presub{n}\pi_{\ast}(k,i) 
< \varepsilon
\quad \mbox{for all $n \in \bbN$},
\end{equation*}
which shows that
$\presub{n}\vc{\Pi}_{\ast}:=\{\presub{n}\vc{\pi}_{\ast};n\in\bbN\}$ is
{\it tight} and thus {\it relatively compact} (see, e.g., \cite[Theorem 5.1]{Bill99}), i.e., there exists a convergent subsequence of
$\presub{n}\vc{\Pi}_{\ast}$. 

Let $\{\presub{n_m}\vc{\pi}_{\ast};m\in\bbZ_+\}$
denote an arbitrary convergent subsequence of
$\presub{n}\vc{\Pi}_{\ast}$ such that
\begin{equation}
\lim_{m\to\infty}\presub{n_m}\vc{\pi}_{\ast} = \vc{\pi}_{\ast},
\label{lim-(n_m)pi_*}
\end{equation}
where $\vc{\pi}_{\ast}:=(\pi_{\ast}(k,i))_{(k,i) \in \bbF}$ is a
probability vector. Furthermore, let
$\presub{n}\vc{P}_{\ast}^{(t)}:= \break (\presub{n}p_{\ast}^{(t)}(k,i;\ell,j))_{(k,i),(\ell,j)
  \in \bbF}$ denote the transition matrix function of the
infinitesimal generator $\presub{n}\vc{Q}_{\ast}$. By definition,
\begin{equation}
\presub{n_m}\vc{\pi}_{\ast}\, \presub{n_m}\vc{P}_{\ast}^{(t)}
= \presub{n_m}\vc{\pi}_{\ast},\qquad t \ge 0.
\label{eqn-(n_m)pi_*}
\end{equation}
It also follows from \cite[Theorem 2.1 and Remark 2.2]{Hart12} that
\begin{equation}
\lim_{n\to\infty}\presub{n}\vc{P}_{\ast}^{(t)} = \vc{P}^{(t)},
\qquad t \ge 0.
\label{lim-(n)P_0^{(t)}}
\end{equation}
Applying Fatou's lemma, (\ref{lim-(n_m)pi_*}) and
(\ref{lim-(n)P_0^{(t)}}) to (\ref{eqn-(n_m)pi_*}), we have
\begin{equation}
\vc{\pi}_{\ast} \vc{P}^{(t)}
\le \vc{\pi}_{\ast}, \qquad t \ge 0,
\label{ineqn-pi_*}
\end{equation}
which shows that $\vc{\pi}_{\ast}$ is the subinvariant probability
vector of $\vc{P}^{(t)}$.  Recall here that $\vc{Q}$ is ergodic. The
ergodicity of $\vc{Q}$ together with (\ref{ineqn-pi_*}) implies that
$\vc{\pi}_{\ast}$ is the unique probability vector satisfying
$\vc{\pi}_{\ast} \vc{P}^{(t)} = \vc{\pi}_{\ast}$ \cite[Section 5.2,
  Theorem 2.8]{Ande91}. Therefore, $\vc{\pi}_{\ast} = \vc{\pi}$ (see
Definition~\ref{defn-pi} and Remark~\ref{rem-pi}).  This result and
\cite[Corollary of Theorem 5.1]{Bill99} yield
\[
\lim_{n\to\infty}\presub{n}\vc{\pi}_{\ast} = \vc{\pi}.
\]
As a result, we have (\ref{lim-(n)pi_n}).
\end{proof}

Theorem~\ref{thm-block-augmentation} shows that
$\{\presub{n}\vc{\pi}_n\}$, $\{\presub{n}\vc{\pi}_0\}$ and
$\{\presub{n}\vc{\pi}_{\ast}\}$ can be approximations to
$\vc{\pi}$. In addition, it follows from (\ref{order-pi's}) that, for
all $m \in \bbZ_+$,
\begin{eqnarray}
0 \le 
\sum_{k=0}^m \sum_{i\in\bbD} \{ \presub{n}\pi_n(k,i) - \pi(k,i) \}
&\le& \sum_{k=0}^m \sum_{i\in\bbD} \{ \presub{n}\pi_{\ast}(k,i) -\pi(k,i) \}
\nonumber
\qquad
\\
&\le& \sum_{k=0}^m \sum_{i\in\bbD} \{ \presub{n}\pi_0(k,i) - \pi(k,i) \}.
\label{ineqn-pi's}
\end{eqnarray}
Therefore, we can say that the LC- (resp.\ FC-) block-augmented truncation of
$\vc{Q} \in \sfBM_d$ is the best (resp.\ worst) among all the block-augmented
truncations of $\vc{Q}$ in the sense shown in (\ref{ineqn-pi's}).

\section{Error bounds for last-column-block-augmented truncations}\label{sec-bounds}

In the previous section, we already have shown that the stationary
distribution vector $\{\presub{n}\vc{\pi}_n\}$ of the
LC-block-augmented truncation $\presub{n}\vc{Q}_n$ is the best
approximation to the stationary distribution vector $\vc{\pi}$ of
$\vc{Q}\in \sfBM_d$. In this section, we do not necessarily assume
$\vc{Q} \in \sfBM_d$, but assume that $\vc{Q}$ is block-wise dominated
by another generator $\wt{\vc{Q}} \in \sfBM_d$, which is possibly
equal to $\vc{Q}$. We then present upper bounds for the total
variation distance between $\presub{n}\vc{\pi}_n$ and $\vc{\pi}$,
i.e.,
\[
\| \presub{n}\vc{\pi}_n - \vc{\pi} \|
:= \sum_{(k,i)\in\bbF} |\presub{n}\pi_n(k,i) - \pi(k,i) |,
\]
where $\| \cdot \|$ denotes the total variation norm. 

We begin with the introduction of some definitions and assumptions for
the subsequent discussion. For any $1 \times |\bbF|$ vector
$\vc{x}=(x(k,i))_{(k,i)\in\bbF}$, let $\left\| \vc{x}
\right\|_{\svc{v}}$ denote
\[
\left\| \vc{x} \right\|_{\svc{v}} 
:= \sup_{|\svc{g}| \le \svc{v}} 
\left| \sum_{(k,i)\in\bbF}x(k,i)g(k,i) \right|
= \sup_{\svc{0} \le \svc{g} \le \svc{v}}\sum_{(k,i)\in\bbF}|x(k,i)|g(k,i),
\]
where $|\vc{g}|$ is a column vector obtained by taking the absolute
value of each element of $\vc{g}$. The quantity $\left\| \vc{x}
\right\|_{\svc{v}}$ is called the {\it $\vc{v}$-norm of
  $\vc{x}$}. Note here that $\| \cdot \|_{\vc{e}} = \| \cdot \|$,
i.e., the $\vc{e}$-norm is equivalent to the total variation norm. Let
$\dd{I}_{\{ \cdot \}}$ denote a function that takes value one if the
statement in the braces is true and takes value zero otherwise. For $K
\in \bbZ_+$, let $\vc{1}_K=(1_K(k,i))_{(k,i)\in\bbF}$ denote a column
vector such that
\[
1_K(k,i)= \left\{
\begin{array}{ll}
1, & (k,i) \in \bbF^{\leqslant K},
\\
0, & (k,i) \in \bbF^{> K}.
\end{array}
\right.
\]
Furthermore, let $\{(\presub{n}X_t,\presub{n}J_t);t\ge0\}$ denote a
Markov chain with infinitesimal generator $\presub{n}\vc{Q}_n$, and
let $\presub{n}\vc{p}_n^{(t)}(k,i) =
(\presub{n}p_n^{(t)}(k,i;\ell,j))_{(\ell,j)\in\bbF}$ denote a probability
vector such that
\begin{equation*}
\presub{n}p_n^{(t)}(k,i;\ell,j) 
=\PP(\presub{n}X_t = \ell, \presub{n}J_t = j 
\mid \presub{n}X_0 = k, \presub{n}J_0 = i).
\end{equation*}
It follows from (\ref{def-trunc{q}_n(k,i;l,j)}) that 
\begin{equation}
\presub{n}p_n^{(t)}(k,i;\ell,j) = 0,
\qquad t \ge 0,\ (k,i) \in \bbF^{\leqslant n},\ (\ell,j) \in \bbF^{>n}.
\label{eqn-(n)p_n^{(u)}=0}
\end{equation}
Finally, we make the following assumptions.
\begin{assumpt}\label{assumpt-basic}
(i) $\vc{Q} \prec_d \wt{\vc{Q}}$; (ii) $\wt{\vc{Q}} \in \sfBM_d$ and
  $\wt{\vc{Q}}$ is irreducible.
\end{assumpt}

\begin{assumpt}\label{assumpt-geo}
There exist some constants $c,b \in (0, \infty)$ and column vector
$\vc{v} = \break (v(k,i))_{(k,i)\in\bbF} \in \sfBI_d$ with $\vc{v} \ge \vc{e}$
such that
\begin{equation}
\wt{\vc{Q}}\vc{v} \le -c \vc{v} + b\vc{1}_0.
\label{ineqn-hat{Q}v}
\end{equation}
\end{assumpt}

Clearly, Assumption~\ref{assumpt-basic} implies
Assumption~\ref{assumpt-Q_1-Q_2}. In addition,
Assumption~\ref{assumpt-geo}, together with
Assumption~\ref{assumpt-basic}, implies that $\wt{\vc{Q}}$ is
exponentially ergodic \cite[Theorem 20.3.2]{Meyn09}. Thus, it follows
from Lemma~\ref{lem-Q-4} that $\wt{\vc{Q}}$ and $\vc{Q}$ have the
unique stationary distribution vectors $\wt{\vc{\pi}}$ and $\vc{\pi}$,
respectively, such that $\vc{\pi} \prec_d\wt{\vc{\pi}}$. We now define
$\vc{\varpi}=(\varpi(i))_{i\in\bbD}$ as a $1 \times d$ probability
vector such that $\varpi(i) = \sum_{k=0}^{\infty}\pi(k,i)$ for $i \in
\bbD$. It then follows from Lemma~\ref{lem-Q-2} that $\vc{\varpi}$ is
the common stationary distribution vector of the Markov chains
$\{J_t\}$ and $\{\wt{J}_t\}$ and
\begin{equation}
\varpi(i) = \sum_{k=0}^{\infty}\pi(k,i) 
= \sum_{k=0}^{\infty}\wt{\pi}(k,i),
\qquad i \in \bbD.
\label{eqn-varpi}
\end{equation}

In what follows, we estimate $\| \presub{n}\vc{\pi}_n -
\vc{\pi}\|$. By the triangular inequality,
\begin{eqnarray}
\left\| \presub{n}\vc{\pi}_n - \vc{\pi} \right\|
&\le&
\left\| \vc{p}^{(t)}(0,\vc{\varpi}) - \vc{\pi} \right\|
+ \left\| \presub{n}\vc{p}_n^{(t)}(0,\vc{\varpi}) - \presub{n}\vc{\pi}_n \right\|
\nonumber
\\
&& ~~{} + 
\left\| \presub{n}\vc{p}_n^{(t)}(0,\vc{\varpi}) - \vc{p}^{(t)}(0,\vc{\varpi}) 
\right\|,~~~~~n \in \bbN,\ t \ge 0,
\label{eqn-27}
\end{eqnarray}
where $\vc{p}^{(t)}(k,i) = (p^{(t)}(k,i;\ell,j))_{(\ell,j)\in\bbF}$
and, for any function $\varphi$ on $\bbF$, 
\[
\varphi(k,\vc{\varpi}) =
\sum_{i\in\bbD} \varpi(i)\varphi(k,i), \quad k \in \bbZ_+.
\]

The following lemma provides upper bounds for the first and second
terms in the right hand side of (\ref{eqn-27}).
\begin{lem}\label{lem-geo-2}
If Assumptions~\ref{assumpt-basic} and \ref{assumpt-geo} hold, then
the following inequalities hold for all $k \in \bbZ_+$ and $t \ge 0$:
\begin{equation}
\left\|  \vc{p}^{(t)}(k,\vc{\varpi}) - \vc{\pi}  \right\|_{\svc{v}}
\le 2\rme^{-ct} \left[v(k,\vc{\varpi})(1 - 1_0(k,\vc{\varpi})) 
+ b/c
\right],
\label{eqn-33a}
\end{equation}
and, for $~n \in \bbN$, 
\begin{eqnarray}
\left\|  
\presub{n}\vc{p}_n^{(t)}(k,\vc{\varpi}) - \presub{n}\vc{\pi}_n  
\right\|_{\svc{v}}
&\le& 2\rme^{-ct} \left[v(k,\vc{\varpi})(1 - 1_0(k,\vc{\varpi})) 
+ b/c
\right].
\label{eqn-33b}
\end{eqnarray}
\end{lem}

\begin{proof}
We first prove (\ref{eqn-33a}). For this purpose, we provide
definitions and notation.  Let $(X,J)$ and $(\wt{X},\wt{J})$ denote
two random vectors on a probability space $(\Omega,\calF,\PP)$ such
that
\begin{align*}
&&&& \PP(X=k, J=i) 
&= \pi(k,i), & (k,i) & \in \bbF, &&&&
\\
&&&&\PP(\wt{X}=k, \wt{J}=i) 
&= \wt{\pi}(k,i), & (k,i) & \in \bbF. &&&&
\end{align*}
It follows from (\ref{eqn-varpi}) and $\vc{\pi} \prec_d
\wt{\vc{\pi}}$ that
\begin{align*}
\PP(J = i) 
&= \sum_{\ell=0}^{\infty}\pi(\ell,i)
= \varpi(i)
=   \sum_{\ell=0}^{\infty}\wt{\pi}(\ell,i)
=   \PP(\wt{J} = i), & i &\in \bbD,
\\
\PP(X > k \mid J = i) 
&=  { \sum_{\ell=k+1}^{\infty}\pi(\ell,i) \over \varpi(i) }
\nonumber
\\
&\le { \sum_{\ell=k+1}^{\infty}\wt{\pi}(\ell,i) \over \varpi(i) }
=    \PP(\wt{X} > k \mid \wt{J} = i), & k &\in \bbZ_+.
\end{align*}
Thus, we assume, without loss of generality, that $X \le \wt{X}$ and $J
= \wt{J}$ \cite[Theorem 1.2.4]{Mull02}.

We now define $\{(\wt{X}_t^{(h)},\wt{J}_t^{(h)});t\ge0\}$, $h \in
\{0,1,2\}$, as Markov chains with infinitesimal generator
$\wt{\vc{Q}}$ on the probability space $(\Omega,\calF,\PP)$ such that
\[
(\wt{X}_0^{(0)},\wt{J}_0^{(0)}) =(0,\wt{J}),  \quad 
(\wt{X}_0^{(1)},\wt{J}_0^{(1)}) = (k,\wt{J}), \quad 
(\wt{X}_0^{(2)},\wt{J}_0^{(2)}) = (\wt{X},\wt{J}).
\]
We also define $\{(X_t^{(h)},J_t^{(h)});t\ge0\}$, $h \in \{0,1,2\}$,
as Markov chains with infinitesimal generator $\vc{Q}$ on the
probability space $(\Omega,\calF,\PP)$ such that
\[
(X_0^{(0)},J_0^{(0)}) =(0,J),  \quad 
(X_0^{(1)},J_0^{(1)}) = (k,J), \quad 
(X_0^{(2)},J_0^{(2)}) = (X,J).
\]
Recall here that $\vc{Q} \prec_d \wt{\vc{Q}}$ and $\wt{\vc{Q}} \in
\sfBM_d$. Therefore, according to
Lemmas~\ref{lem1-continuous-ordering} and
\ref{lem2-continuous-ordering}, we assume that, for each $h \in
\{0,1,2\}$,
\begin{equation}
X_t^{(h)} \le \wt{X}_t^{(h)},\quad
J_t^{(h)} = \wt{J}_t^{(h)}\quad 
\mbox{for all $t \ge 0$},
\label{ordering-01}
\end{equation}
and that
\begin{equation}
\wt{X}_t^{(0)} \le \wt{X}_t^{(1)}, \quad 
\wt{X}_t^{(0)} \le \wt{X}_t^{(2)}, \quad
\wt{J}_t^{(0)} = \wt{J}_t^{(1)} = \wt{J}_t^{(2)}
\quad 
\mbox{for all $t \ge 0$}.
\label{ordering-02}
\end{equation}
For the Markov chains $\{(X_t^{(h)},J_t^{(h)})\}$'s, $h \in
\{0,1,2\}$, we introduce two coupling times $T^{(1)}$ and $T^{(2)}$ as
follows:
\begin{eqnarray*}
T^{(1)} &=& \inf\{t \ge 0: X_t^{(1)}=X_t^{(0)}\},
\\
T^{(2)} &=& \inf\{t \ge 0: X_t^{(2)}=X_t^{(0)}\}.
\end{eqnarray*}
We then assume, without loss of generality (see Remark~\ref{rem-first-meeting-lasts-forever}), that
\begin{eqnarray*}
X_t^{(1)} &=& X_t^{(0)}\qquad \mbox{for all $t \ge T^{(1)}$},
\\
X_t^{(2)} &=& X_t^{(0)}\qquad \mbox{for all $t \ge T^{(2)}$}.
\end{eqnarray*}
For convenience, we also introduce the following notation:
\begin{align*}
\EE_{(k,i)}[\,\, \cdot \,\,] 
&= \EE[~ \cdot \mid X_0=k, J_0=i], 
\\
\EE_{(k,i);(0,j)}[\,\, \cdot \,\,] 
&= \EE[~  \cdot \mid (X_0^{(h)}, 
J_0^{(h)}) = (k,i), (X_0^{(0)},J_0^{(0)}) = (0,j)],
\end{align*}
where $h \in \{1,2\}$, $(k,i) \in \bbF$ and $j\in\bbD$. 

We are now ready to prove (\ref{eqn-33a}). Proceeding as in the derivation of \cite[Eq.\ (3.18)]{Masu15-ADV} (see also \cite[Eq.\ (3.6)]{Lund96}), we have, for $|\vc{g}| \le \vc{v}$,
\begin{eqnarray}
|\vc{p}^{(t)}(k,\vc{\varpi})\vc{g} - \vc{\pi}\vc{g}|
&\le& 
\EE\!\left[\EE_{(k,J);(0,J)}[v(X_t^{(1)},J_t^{(1)}) \cdot \dd{I}_{\{T^{(1)} > t\}}] \right]
\nonumber
\\
&& {} 
+ \EE\!\left[\EE_{(k,J);(0,J)}[v(X_t^{(0)},J_t^{(0)}) \cdot \dd{I}_{\{T^{(1)} > t\}}]\right]
\nonumber
\\
&& {} + 
\EE\!\left[\EE_{(X,J);(0,J)}[v(X_t^{(2)},J_t^{(2)}) \cdot \dd{I}_{\{T^{(2)} > t\}}] \right]
\nonumber
\\
&& {} 
+ \EE\!\left[\EE_{(X,J);(0,J)}[v(X_t^{(0)},J_t^{(0)}) \cdot \dd{I}_{\{T^{(2)} > t\}}]\right]. \qquad
\label{eqn-06a}
\end{eqnarray}
Combining (\ref{eqn-06a}) with 
 (\ref{ordering-01}), (\ref{ordering-02}) and $\vc{v} \in
\sfBI_d$, we obtain, for $|\vc{g}| \le \vc{v}$,
\begin{eqnarray}
|\vc{p}^{(t)}(k,\vc{\varpi})\vc{g} - \vc{\pi}\vc{g}|
&\le& 2\EE\!\left[
\EE_{(k,J);(0,J)}[v(\wt{X}_t^{(1)},\wt{J}_t^{(1)}) 
\cdot \dd{I}_{\{T^{(1)} > t\}}] 
\right]
\nonumber
\\
&& {} + 2 \EE\!\left[
\EE_{(X,J);(0,J)}[v(\wt{X}_t^{(2)},\wt{J}_t^{(2)}) 
\cdot \dd{I}_{\{T^{(2)} > t\}}] 
\right]. \qquad
\label{add-eqn-151022-01}
\end{eqnarray}
Furthermore, it follows from (\ref{ordering-01}) and
(\ref{ordering-02}) that, for each $h \in \{1,2\}$, $\{\wt{X}_t^{(h)}
= 0\}$ implies $\{ X_t^{(h)} = X_t^{(0)} = 0\}$, which leads to $T^{(h)} \le
\inf\{t \ge 0; \wt{X}_t^{(h)} = 0\}$. Therefore,
(\ref{add-eqn-151022-01}) yields
\begin{eqnarray}
\left\|  \vc{p}^{(t)}(k,\vc{\varpi}) - \vc{\pi} \right\|_{\svc{v}}
&\le& 2\EE\!\left[\EE_{(k,\wt{J})}[v(\wt{X}_t,\wt{J}_t) 
\cdot \dd{I}_{\{\wt{\tau}_0 > t\}}] \right]
\nonumber
\\
&& {}
+  2\EE\!\left[\EE_{(\wt{X},\wt{J})}[v(\wt{X}_t,\wt{J}_t) 
\cdot \dd{I}_{\{\wt{\tau}_0 > t\}}] \right],
\label{eqn-20}
\end{eqnarray}
where $\wt{\tau}_0 = \inf\{t \ge 0: \wt{X}_t = 0\}$.

Let $M_t=\rme^{ct}v(\wt{X}_t,\wt{J}_t)\dd{I}_{\{\wt{\tau}_0 > t\}}$
for $t \ge 0$. It is shown in Lemma~\ref{prop-supermartingale} that
$\{M_t\}$ is a supermartingale.  Let $\{\theta_{\nu};\nu\in \bbZ_+\}$
denote a sequence of stopping times for $\{M_t;t\ge0\}$ such that $0
\le \theta_1 \le \theta_2 \le \cdots$ and
$\lim_{\nu\to\infty}\theta_{\nu}=\infty$. It then follows that, for
any $u \ge 0$, $\min(u,\theta_{\nu})$ is a stopping time for
$\{M_t;t\in\bbZ_+\}$. Therefore, using Doob's optional sampling
theorem (see, e.g., \cite[Section~10.10]{Will91}), we have
$\EE_{(k,i)}[M_{\min(t,\theta_{\nu})}] \le \EE_{(k,i)}[M_0]$, i.e.,
\begin{eqnarray*}
\EE_{(k,i)}[\rme^{c\min(t,\theta_{\nu})}
v(\wt{X}_{\min(t,\theta_{\nu})},\wt{J}_{\min(t,\theta_{\nu})})
\dd{I}_{\{\wt{\tau}_0 > \min(t,\theta_{\nu}) \}}]
 \le v(k,i)(1 - 1_0(k,i)),
\end{eqnarray*}
for $(k,i) \in \bbF$.
Letting $\nu\to\infty$ in the above inequality and using Fatou's
lemma, we obtain
\begin{equation}
\EE_{(k,i)}[v(\wt{X}_t,\wt{J}_t) \dd{I}_{\{\wt{\tau}_0 > t \}}]
\le \rme^{-ct}   v(k,i)(1 - 1_{0}(k,i)),
\qquad (k,i) \in \bbF,
\label{eqn-10}
\end{equation}
and thus
\begin{eqnarray}
\EE\!\left
[\EE_{(k,\wt{J})}[v(\wt{X}_t,\wt{J}_t)\dd{I}_{\{\wt{\tau}_0 > t \}}] 
\right]
&=& \sum_{i\in\bbD}\varpi(i)
\EE_{(k,i)}[v(\wt{X}_t,\wt{J}_t) \dd{I}_{\{\wt{\tau}_0 > t \}}]
\nonumber
\\
&\le& \rme^{-ct} v(k,\vc{\varpi})(1 - 1_{0}(k,\vc{\varpi})),
\quad k \in \bbZ_+, \qquad
\label{eqn-11}
\end{eqnarray}
where we use $1_0(k,i) = 1_0(k,\vc{\varpi})$ for all $i \in
\bbD$. Furthermore, pre-multiplying both sides of
(\ref{ineqn-hat{Q}v}) by $\wt{\vc{\pi}}$, we have $\wt{\vc{\pi}}\vc{v}
\le b / c$. Combining this and (\ref{eqn-10}) yields
\begin{equation}
\EE\!\left[
\EE_{(\wt{X},\wt{J})}[v(\wt{X}_t,\wt{J}_t) \cdot \dd{I}_{\{\wt{\tau}_0 > t\}}] 
\right]
\le \rme^{-ct}  \sum_{(k,i)\in\bbF}\pi(k,i)v(k,i)
\le  \rme^{-ct} {b\over c} .
\label{eqn-12}
\end{equation}
Substituting (\ref{eqn-11}) and (\ref{eqn-12}) into (\ref{eqn-20}), we
obtain (\ref{eqn-33a}).

Next, we prove (\ref{eqn-33b}). Let
$\presub{n}\wt{\vc{Q}}_n:=(\presub{n}\wt{q}_n(k,i;\ell,j))_{(k,i),(\ell,j)\in\bbF}$
denote the LC-block-augmented truncation of $\wt{\vc{Q}}$, which is
defined in a similar way to (\ref{def-trunc{q}_n(k,i;l,j)}).  Since
$\vc{Q} \prec_d \wt{\vc{Q}}$ and $\wt{\vc{Q}} \in \sfBM_d$, we have
$\presub{n}\wt{\vc{Q}}_n \in \sfBM_d$ and
\[
\presub{n}\vc{Q}_n \prec_d \presub{n}\wt{\vc{Q}}_n \prec_d \wt{\vc{Q}}.
\]
It follows from $\presub{n}\wt{\vc{Q}}_n \prec_d \wt{\vc{Q}}$, $\vc{v}
\in \sfBI_d$ and (\ref{ineqn-hat{Q}v}) that
\begin{equation}
\presub{n}\wt{\vc{Q}}_n\vc{v} \le \wt{\vc{Q}}\vc{v} 
\le -c\vc{v} + b\vc{1}_0.
\label{ineqn-(n)P_n*v}
\end{equation}
Therefore, we can readily prove (\ref{eqn-33b}) by replacing
$\wt{\vc{Q}}$ and $\vc{Q}$ with $\presub{n}\wt{\vc{Q}}_n$ and
$\presub{n}\vc{Q}_n$, respectively, in the proof of
(\ref{eqn-33a}). The details are omitted.
\end{proof}

\medskip

According to Lemma~\ref{lem-geo-2}, it remains to estimate the third
term in the right hand side of (\ref{eqn-27}) in order to derive an
upper bound for $\| \presub{n}\vc{\pi}_n - \vc{\pi}\|$. By doing this,
we obtain the following theorem.
\begin{thm}\label{thm-continu-geo}
If Assumptions~\ref{assumpt-basic} and \ref{assumpt-geo} hold, then, 
for all $n\in \bbN$ and $t \ge 0$,
\begin{eqnarray}
\left\| \presub{n}\vc{\pi}_n - \vc{\pi} \right\|
&\le& {b \over c}
\left(
4\rme^{-ct}  + 2t \sum_{j\in\bbD}{|\wt{q}(n,j;n,j)| \over v(n,j)}
\right).
\label{continu-geo-bound}
\end{eqnarray}
\end{thm}

\begin{rem}\label{rem-minimum-bound}
It is easy to see that, for each $n \in \bbN$, the right hand side of
(\ref{continu-geo-bound}) takes the minimum value at $t =
c^{-1}t^{\ast}(n)$, where
\begin{equation}
t^{\ast}(n)
= \max\left\{
-
\log\left( 
{1 \over 2c} \sum_{j\in\bbD}{|\wt{q}(n,j;n,j)| \over v(n,j)}
\right)
,0
\right\},
\qquad n \in \bbN.
\label{defn-t^*(n)}
\end{equation}
Substituting (\ref{defn-t^*(n)}) into (\ref{continu-geo-bound}) yields
\begin{eqnarray}
\left\| \presub{n}\vc{\pi}_n - \vc{\pi} \right\|
&\le& {4b \over c} 
(t^{\ast}(n) + 1 ) \exp\{-t^{\ast}(n)\},\qquad n \in \bbN.
\label{continu-geo-bound-minimun}
\end{eqnarray}
If $\lim_{n\to\infty} \sum_{j\in\bbD}|\wt{q}(n,j;n,j)| / v(n,j) = 0$
and thus $\lim_{n\to\infty}t^{\ast}(n)=\infty$, then the right hand
side of (\ref{continu-geo-bound-minimun}) converges to zero as $n \to \infty$.
A similar discussion is found in \cite{Twee98} (see Eq.~(50) therein).
\end{rem}

\begin{proofof}{Theorem~\ref{thm-continu-geo}} 
Letting $k=0$ and $\vc{v} = \vc{e}$ in (\ref{eqn-33a}) and
(\ref{eqn-33b}) (see Lemma~\ref{lem-geo-2}) and substituting the
result into (\ref{eqn-27}), we have
\[
\left\| \presub{n}\vc{\pi}_n - \vc{\pi} \right\|
\le {4b\rme^{-ct} \over c}
+ \left\| 
\presub{n}\vc{p}_n^{(t)}(0,\vc{\varpi}) - \vc{p}^{(t)}(0,\vc{\varpi}) 
\right\|.
\]
Therefore, it suffices to prove that, for all $n\in \bbN$ and $t \ge
0$,
\begin{equation}
\left\| 
\presub{n}\vc{p}_n^{(t)}(0,\vc{\varpi}) - \vc{p}^{(t)}(0,\vc{\varpi}) 
\right\|
\le {2tb \over c}\sum_{j\in\bbD}{|\wt{q}(n,j;n,j)| \over v(n,j)}.
\label{add-eqn-23}
\end{equation}
Note here that all the off-diagonal elements of $\vc{Q}$ are
nonnegative.  Using this fact, (\ref{eqn-(n)p_n^{(u)}=0}) and
Lemma~\ref{appen-lem-2}, we have
\begin{eqnarray}
\lefteqn{
\left\| \presub{n}\vc{p}_n^{(t)}(0,\vc{\varpi}) - \vc{p}^{(t)}(0,\vc{\varpi}) 
\right\|
}
~~&&
\nonumber\\
&\le& \sum_{i \in \bbD}\varpi(i)
\left\| \presub{n}\vc{p}_n^{(t)}(0,i) - \vc{p}^{(t)}(0,i) 
\right\|
\nonumber
\\
&\le&  2\int_0^t 
\sum_{(\ell,j)\in\bbF} 
\left(\sum_{i \in \bbD}\varpi(i)\presub{n}p_n^{(u)}(0,i;\ell,j) \right) \rmd u 
\sum_{ (\ell',j') \in \bbF^{>n} }  |q(\ell,j;\ell',j')|
\nonumber
\\
&=& 2\int_0^t
\sum_{(\ell,j)\in\bbF^{\leqslant n}} 
\left(\sum_{i \in \bbD}\varpi(i)\presub{n}p_n^{(u)}(0,i;\ell,j) \right) \rmd u
\sum_{ (\ell',j') \in \bbF^{>n} } q(\ell,j;\ell',j').\qquad~~
\label{eqn-29}
\end{eqnarray}
In addition, $\vc{Q} \prec_d \wt{\vc{Q}}$ yields
\[
\sum_{ (\ell',j') \in \bbF^{>n} } q(\ell,j;\ell',j')
\le \sum_{ (\ell',j') \in \bbF^{>n} } \wt{q}(\ell,j;\ell',j'),
\qquad (\ell,i) \in \bbF^{\leqslant n}.
\]
Substituting this inequality into (\ref{eqn-29}) and using
(\ref{eqn-(n)p_n^{(u)}=0}), we obtain
\begin{eqnarray}
\lefteqn{
\left\| \presub{n}\vc{p}_n^{(t)}(0,\vc{\varpi}) - \vc{p}^{(t)}(0,\vc{\varpi}) 
\right\|
}
\quad &&
\nonumber
\\
&\le& 2\int_0^t
\sum_{(\ell,j)\in\bbF^{\leqslant n}} 
\left(\sum_{i \in \bbD}\varpi(i)\presub{n}p_n^{(u)}(0,i;\ell,j) \right)\rmd u
\sum_{ (\ell',j') \in \bbF^{>n} } \wt{q}(\ell,j;\ell',j') 
\nonumber
\\
&=& 2\int_0^t
\sum_{(\ell,j)\in\bbF} 
\left(\sum_{i \in \bbD}\varpi(i)\presub{n}p_n^{(u)}(0,i;\ell,j) \right)\rmd u
\sum_{ (\ell',j') \in \bbF^{>n} } \wt{q}(\ell,j;\ell',j') 
\nonumber
\\
&=& 2\int_0^t
(\vc{\varpi},\vc{0},\vc{0},\dots,) \presub{n}\vc{P}_n^{(u)}  \vc{a}_n \,\rmd u,
\label{eqn-30}
\end{eqnarray}
where $\vc{a}_n := (a_n(\ell,j))_{(\ell,j)\in\bbF}$ is a column vector such
that
\begin{equation}
a_n(\ell,j) = 
\left\{
\begin{array}{ll}
\dm\sum_{ (\ell',j') \in \bbF^{>n} } \wt{q}(\ell,j;\ell',j'), 
& (\ell,j) \in \bbF^{\leqslant n},
\\
\dm\sum_{ (\ell',j') \in \bbF^{>n} } \wt{q}(n,j;\ell',j'), 
& (\ell,j) \in \bbF^{>n}.
\end{array}
\right.
\label{defn-a_n}
\end{equation}
It follows from (\ref{defn-a_n}) and $\wt{\vc{Q}} \in \sfBM_d$ that
$\vc{a}_n \in \sfBI_d$. It also follows from (\ref{defn-a_n}) and the
definition of $\presub{n}\wt{\vc{Q}}_n$ that, for $(\ell,j) \in
\bbF^{\leqslant n}$,
\begin{eqnarray*}
\sum_{j'\in\bbD}\presub{n}\wt{q}_n(\ell,j;n,j') 
&=& \sum_{j'\in\bbD}\wt{q}(\ell,j;n,j') 
+ \sum_{(\ell',j') \in \bbF^{>n}} \wt{q}(\ell,j;\ell',j')
\nonumber
\\
&=& \sum_{j'\in\bbD}\wt{q}(\ell,j;n,j') + a_n(\ell,j),
\end{eqnarray*}
which leads to
\begin{equation}
a_n(\ell,j) = \sum_{j'\in\bbD} \left(
\presub{n}\wt{q}_n(\ell,j;n,j') - \wt{q}(\ell,j;n,j')
\right),\qquad (\ell,j) \in \bbF^{\leqslant n}.
\label{eqn-a}
\end{equation}

We now define
$\presub{n}\wt{\vc{\pi}}_n:=(\presub{n}\wt{\pi}_n(k,i))_{(k,i)\in\bbF}$
as the stationary distribution vector of $\presub{n}\wt{\vc{Q}}_n$. It
then follows from Lemma~\ref{lem-(n)pi_*(k,i)} and the irreducibility
of $\wt{\vc{Q}}$ that
\begin{equation}
\presub{n}\wt{\pi}_n(k,i) = 0,\qquad (k,i) \in \bbF^{>n}.
\label{eqn-(n)widetilde{pi}_*(k,i)=0}
\end{equation}
It also follows from $\vc{Q} \prec_d \wt{\vc{Q}} \in \sfBM_d$ that
$\presub{n}\wt{\vc{Q}}_n \in \sfBM_d$ and $\presub{n}\vc{Q}_n \prec_d
\presub{n}\wt{\vc{Q}}_n \prec_d \wt{\vc{Q}}$. Thus,
Lemmas~\ref{lem-Q}, \ref{lem-Q-2} and \ref{lem-Q-3} imply that
\begin{align}
&&&&
\sum_{k=0}^{\infty}\presub{n}\wt{\pi}_n(k,i)
&= \sum_{k=0}^{\infty}\wt{\pi}(k,i)
= \varpi(i),
& i &\in \bbD, &&&&
\label{sum-(n)wt{pi}_n(k,i)}
\\
&&&&
\presub{n}\vc{P}_n^{(t)} &\prec_d
\presub{n}\wt{\vc{P}}_n^{(t)} \in \sfBM_d,
& t &\ge 0,&&&&
\label{relation-(n)P_n^{(t)}-(n)wt{P}_n^{(t)}}
\end{align}
where $\presub{n}\vc{P}_n^{(t)}$ and $\presub{n}\wt{\vc{P}}_n^{(t)}$
are the transition matrix functions of the infinitesimal generators
$\presub{n}\vc{Q}_n$ and $\presub{n}\wt{\vc{Q}}_n$,
respectively. Using (\ref{relation-(n)P_n^{(t)}-(n)wt{P}_n^{(t)}}) and
$\vc{a}_n \in \sfBI_d$, we have
\begin{eqnarray*}
\presub{n}\vc{P}_n^{(t)} \vc{a}_n
&\le& \presub{n}\wt{\vc{P}}_n^{(t)} \vc{a}_n,
\qquad t \ge 0.
\end{eqnarray*}
Applying this inequality to (\ref{eqn-30}) yields
\begin{eqnarray}
\left\| \presub{n}\vc{p}_n^{(t)}(0,\vc{\varpi}) - \vc{p}^{(t)}(0,\vc{\varpi}) 
\right\|
\le 2\int_0^t
(\vc{\varpi},\vc{0},\vc{0},\dots,) \presub{n}\wt{\vc{P}}_n^{(u)}  \vc{a}_n \,\rmd u.
\label{eqn-30-wt}
\end{eqnarray}
Note here that (\ref{sum-(n)wt{pi}_n(k,i)}) leads to
$(\vc{\varpi},0,0,\dots,) \prec_d
\presub{n}\wt{\vc{\pi}}_n$. Combining this relation with
$\presub{n}\wt{\vc{P}}_n^{(t)} \in \sfBM_d$ (see
Proposition~\ref{prop-2}), we have, for $t \ge 0$,
\begin{equation}
(\vc{\varpi},\vc{0},\vc{0},\dots,) \presub{n}\wt{\vc{P}}_n^{(t)} \prec_d
\presub{n}\wt{\vc{\pi}}_n \,\presub{n}\wt{\vc{P}}_n^{(t)}
= \presub{n}\wt{\vc{\pi}}_n,
\label{add-eqn-14}
\end{equation}
where the last equality follows from
$\presub{n}\wt{\vc{\pi}}_n\,\presub{n}\wt{\vc{P}}_n^{(t)}
= \presub{n}\wt{\vc{\pi}}_n$. Furthermore, using (\ref{add-eqn-14}) and
$\vc{a}_n \in \sfBI_d$, we obtain
\begin{eqnarray*}
(\vc{\varpi},\vc{0},\vc{0},\dots,)\presub{n}\wt{\vc{P}}_n^{(t)} \vc{a}_n
&\le& \presub{n}\wt{\vc{\pi}}_n \vc{a}_n,
\qquad t \ge 0.
\end{eqnarray*}
Substituting this into (\ref{eqn-30-wt}) results in
\begin{eqnarray}
\left\| \presub{n}\vc{p}_n^{(t)}(0,\vc{\varpi}) - \vc{p}^{(t)}(0,\vc{\varpi}) 
\right\|
\le  2t \presub{n}\wt{\vc{\pi}}_n \vc{a}_n,
\qquad n \in \bbN,\ t \ge 0.
\label{eqn-30a}
\end{eqnarray}

In what follows, we estimate $\presub{n}\wt{\vc{\pi}}_n \vc{a}_n$.
From (\ref{eqn-a}), (\ref{eqn-(n)widetilde{pi}_*(k,i)=0}) and
$\presub{n}\wt{\vc{\pi}}_n\,\presub{n}\wt{\vc{Q}}_n =
\vc{0}$, we have
\begin{eqnarray}
\presub{n}\wt{\vc{\pi}}_n \vc{a}_n
&=& \sum_{(\ell,j)\in\bbF^{\leqslant n}} 
\presub{n}\wt{\pi}_n(\ell,j)
\sum_{j'\in\bbD} \left(\presub{n}\wt{q}_n(\ell,j;n,j') 
 -  \wt{q}(\ell,j;n,j') \right)
\nonumber
\\
&=& \sum_{(\ell,j)\in\bbF^{\leqslant n}} 
\presub{n}\wt{\pi}_n(\ell,j)
\sum_{j'\in\bbD} \left(  -  \wt{q}(\ell,j;n,j') \right)
\nonumber
\\
&=& \sum_{j'\in\bbD}
\sum_{(\ell,j)\in\bbF^{\leqslant n}} 
\presub{n}\wt{\pi}_n(\ell,j)
\left(  -  \wt{q}(\ell,j;n,j') \right).
\label{eqn-30b}
\end{eqnarray}
By definition, $- \wt{q}(\ell,j;n,j') \le 0$ for
$(\ell,j)\in\bbF^{\leqslant n} \setminus \{(n,j')\}$ and
$-\wt{q}(n,j';n,j') \ge 0$. Therefore,
\[
\sum_{(\ell,j)\in\bbF^{\leqslant n}} 
\presub{n}\wt{\pi}_n(\ell,j)
\left(  -  \wt{q}(\ell,j;n,j') \right)
\le \presub{n}\wt{\pi}_n(n,j') |\wt{q}(n,j';n,j')|,
\quad j' \in \bbD.
\]
Applying this to (\ref{eqn-30b}), we have
\begin{eqnarray}
\presub{n}\wt{\vc{\pi}}_n \vc{a}_n
&\le& \sum_{j'\in\bbD} 
\presub{n}\wt{\pi}_n(n,j') |\wt{q}(n,j';n,j')|.
\label{eqn-30d}
\end{eqnarray}
Pre-multiplying both sides of (\ref{ineqn-(n)P_n*v}) by
$\presub{n}\wt{\vc{\pi}}_n$, we obtain
$\presub{n}\wt{\vc{\pi}}_n\vc{v} \le b/c$, which
leads to
\begin{equation}
\presub{n}\wt{\pi}_n(n,j') \le {b \over c}\cdot {1 \over v(n,j')},
\qquad j' \in \bbD.
\label{ineqn-(n)wt{pi}_n(n,j)}
\end{equation}
Using (\ref{eqn-30d}) and (\ref{ineqn-(n)wt{pi}_n(n,j)}), we have
\[
\presub{n}\wt{\vc{\pi}}_n \vc{a}_n
\le {b \over c} \sum_{j'\in\bbD} {|\wt{q}(n,j';n,j')| \over v(n,j')}.
\]
Substituting this inequality into (\ref{eqn-30a}) yields
(\ref{add-eqn-23}).
\end{proofof}

\medskip

From Theorem~\ref{thm-continu-geo} and Remark~\ref{rem-minimum-bound}, we obtain
Corollary~\ref{thm-continu-02} below, where the drift condition
(\ref{ineqn-hat{Q}v}) for $\wt{\vc{Q}}$ is weakened whereas the set of
states $\{(0,i);i\in\bbD\}$ is assumed to be reachable directly from
each state in $\bbF^{\leqslant K}$ with a sufficiently large $K$.
\begin{coro}\label{thm-continu-02}
Suppose that Assumptions~\ref{assumpt-basic} and \ref{assumpt-geo} hold and there exist some
$c',b' \in (0,\infty)$, $K \in \bbZ_+$ and column vector
$\vc{v}':=(v'(k,i))_{(k,i)\in\bbF} \in \sfBI_d$ with $\vc{v}' \ge
\vc{e}$ such that
\begin{equation}
\wt{\vc{Q}}\vc{v}' \le -c' \vc{v}' + b'\vc{1}_K.
\label{ineqn-hat{Q}v'}
\end{equation}
Let $\wt{\vc{Q}}(k;\ell) = (\wt{q}(k,i;\ell,j))_{(i,j)\in\bbD}$ for
$k,\ell\in\bbZ_+$, and suppose that
\[
\wt{\vc{Q}}(K;0)\vc{e} > \vc{0}.
\]
Under these conditions, the bound (\ref{continu-geo-bound-minimun}),
together with (\ref{defn-t^*(n)}), holds for all $n\in \bbN$, where
$c,b,B \in (0,\infty)$ are constants such that
\begin{eqnarray}
c &=& {c' \over 1+B},
\label{add-defn-c}
\\
b\vc{e} &\ge& b'\vc{e} - B \wt{\vc{Q}}(0;0)\vc{e},
\label{add-defn-b}
\\
B\cdot\wt{\vc{Q}}(K;0)\vc{e} &\ge& b'\vc{e},
\label{add-defn-B}
\end{eqnarray}
and block-wise increasing $\vc{v}=(v(k,i))_{(k,i)\in\bbF}$ is given by
\begin{eqnarray}
v(k,i) &=& 
\left\{
\begin{array}{ll}
v'(0,i), & k=0,~i\in\bbD,
\\
v'(k,i) + B, & k\in\bbN,~i\in\bbD.
\end{array}
\right.
\label{add-defn-v}
\end{eqnarray}
\end{coro}

\proof It suffices to prove that (\ref{ineqn-hat{Q}v}) holds for
$c,b,B \in (0,\infty)$ and $\vc{v} \in \sfBI_d$ such that
(\ref{add-defn-c})--(\ref{add-defn-v}) are satisfied.  We begin with
the estimation of $\sum_{\ell=0}^{\infty}
\wt{\vc{Q}}(0;\ell)\vc{v}(\ell)$.  It follows from (\ref{add-defn-v})
and $\sum_{\ell=0}^{\infty}\wt{\vc{Q}}(k;\ell)\vc{e} = \vc{0}$ for all
$k \in \bbZ_+$ that
\begin{eqnarray}
\sum_{\ell=0}^{\infty}
\wt{\vc{Q}}(k;\ell)\vc{v}(\ell)
&=& \sum_{l=1}^{\infty}\wt{\vc{Q}}(k;\ell)(\vc{v}'(\ell) + B\vc{e})
+ \wt{\vc{Q}}(k;0)\vc{v}'(0)
\nonumber
\\
&=& \sum_{\ell=0}^{\infty}\wt{\vc{Q}}(k;\ell)\vc{v}'(\ell) 
- B \wt{\vc{Q}}(k;0)\vc{e},\qquad k \in \bbZ_+.
\label{add-eqn-07}
\end{eqnarray}
Substituting $k=0$ into (\ref{add-eqn-07}), and applying
(\ref{ineqn-hat{Q}v'}) and $\vc{v}'(0) = \vc{v}(0)$ to the resulting
equation, we have
\begin{eqnarray}
\sum_{\ell=0}^{\infty}
\wt{\vc{Q}}(0;\ell)\vc{v}(\ell)
&\le& -c'\vc{v}(0) + b'\vc{e} - B \wt{\vc{Q}}(0;0)\vc{e}
\nonumber
\\
&\le& -c'\vc{v}(0) + b\vc{e}
\le -c\vc{v}(0) + b\vc{e},
\label{add-eqn-151026-01}
\end{eqnarray}
where the last two inequalities follow from
(\ref{add-defn-b}) and $c' \ge c$ (due to (\ref{add-defn-c})).  

Next we consider $\sum_{\ell=0}^{\infty}
\wt{\vc{Q}}(k;\ell)\vc{v}(\ell)$ for $ k \in \{1,2,\dots,K\}$.  It
follows from $\wt{\vc{Q}} \in \sfBM_d$ that $\wt{\vc{Q}}(k;0)\vc{e}
\ge \wt{\vc{Q}}(K;0)\vc{e}$ for $k=1,2,\dots,K$. Incorporating this
inequality and (\ref{ineqn-hat{Q}v'}) into (\ref{add-eqn-07}) for $k
\in \{1,2,\dots,K\}$, we obtain
\begin{eqnarray}
\sum_{\ell=0}^{\infty}
\wt{\vc{Q}}(k;\ell)\vc{v}(\ell)
&\le& \sum_{\ell=0}^{\infty}\wt{\vc{Q}}(k;\ell)\vc{v}'(\ell) 
- B \wt{\vc{Q}}(K;0)\vc{e}
\nonumber
\\
&\le&
 -c'\vc{v}'(k) + b'\vc{e} 
-B\wt{\vc{Q}}(K;0)\vc{e}
\nonumber
\\
&\le& -c'\vc{v}'(k),\qquad k=1,2,\dots,K, 
\label{add-eqn-04}
\end{eqnarray}
where the last inequality holds due to (\ref{add-defn-B}).  Note here
that (\ref{add-defn-c}), (\ref{add-defn-v}) and $\vc{v}'\ge \vc{e}$
imply
\begin{equation}
c\vc{v}(k) = c (\vc{v}'(k)+B\vc{e}) \le c'\vc{v}'(k),\qquad  k \in \bbN,
\label{add-eqn-05}
\end{equation}
from which and (\ref{add-eqn-04}) we have
\begin{equation}
\sum_{\ell=0}^{\infty} \wt{\vc{Q}}(k;\ell)\vc{v}(\ell)
\le -c\vc{v}(k),\qquad k=1,2,\dots,K.
\label{add-eqn-151026-02}
\end{equation}

Finally, we estimate $\sum_{\ell=0}^{\infty}
\wt{\vc{Q}}(k;\ell)\vc{v}(\ell)$ for $ k \ge K+1$. Substituting
(\ref{ineqn-hat{Q}v'}) into (\ref{add-eqn-07}) for $k \ge K+1$ and
using (\ref{add-eqn-05}) and $\wt{\vc{Q}}(k;0)\vc{e} \ge \vc{0}$, we
obtain
\begin{eqnarray}
\sum_{\ell=0}^{\infty}
\wt{\vc{Q}}(k;\ell)\vc{v}(\ell)
&\le& -c'\vc{v}'(k) - B \wt{\vc{Q}}(k;0)\vc{e}
\nonumber
\\
&\le& -c\vc{v}(k) - B \wt{\vc{Q}}(k;0)\vc{e} 
\nonumber
\\
&\le& -c\vc{v}(k),\qquad k=K+1,K+2,\dots.
\label{add-eqn-151026-03}
\end{eqnarray}
As a result, combining (\ref{add-eqn-151026-01}),
(\ref{add-eqn-151026-02}) and (\ref{add-eqn-151026-03}), we have
(\ref{ineqn-hat{Q}v}).  \qed

\section{Applications}\label{sec-applications}

In this section, we demonstrate the applicability of the error bounds
presented in Section~\ref{sec-bounds}. To this end, we consider a
queue with a batch Markovian arrival process (BMAP) \cite{Luca91} and
level-dependent departure rates.

We first describe the BMAP. The BAMP is controlled by an irreducible
continuous-time Markov chain $\{J(t);t \ge 0\}$ with a finite state
space $\bbD=\{1,2,\dots,d\}$, which is called the background Markov
chain. Let $N(t)$, $t \ge 0$, denote the total number of arrivals in
time interval $(0,t]$, where $N(0) = 0$. We assume that
$\{(N(t),J(t));t\ge0\}$ is a continuous-time Markov chain with state
space $\bbF = \bbZ_+ \times \bbD$ and conservative infinitesimal
generator $\vc{Q}_{\rm BMAP}$ given by
\[
\vc{Q}_{\rm BMAP}
=
\left(
\begin{array}{ccccc}
\vc{D}(0) & \vc{D}(1) & \vc{D}(2) & \vc{D}(3) & \cdots
\\
\vc{O} & \vc{D}(0)   & \vc{D}(1) & \vc{D}(2) & \cdots
\\
\vc{O} & \vc{O}   & \vc{D}(0)   & \vc{D}(1) & \cdots
\\
\vc{O} & \vc{O}   & \vc{O}   & \vc{D}(0)   & \ddots
\\
\vdots & \vdots   &\vdots    & \ddots   & \ddots
\end{array}
\right),
\]
where $\vc{D}(k):=(D_{i,j}(k))_{i,j\in\bbD}$, $k\in\bbN$, is an $d
\times d$ nonnegative matrix and
$\vc{D}(0):=(D_{i,j}(0))_{i,j\in\bbD}$ is a $d \times d$
$q$-matrix. It then follows that, for $t \ge 0$ and $\Delta t \ge 0$,
\begin{eqnarray}
\lefteqn{
\PP(
N(t+\Delta t) - N(t) = k, J(t+\Delta t) =j \mid J(t) = i
)
}
\qquad &&
\nonumber
\\
&=& 
\left\{
\begin{array}{ll}
1 + D_{i,i}(0)\Delta t + o(\Delta t), & k = 0,~i=j \in \bbD,
\\
{}
D_{i,j}(0) \Delta t + o(\Delta t), & k = 0,~i \neq j,~ i,j \in \bbD,
\\
{}
D_{i,j}(k)\Delta t + o(\Delta t), & k \in \bbN,~i,j \in \bbD,
\end{array}
\right.
\label{defn-BMAP}
\end{eqnarray}
where $f(x) = o(g(x))$ represents
$\lim_{x\downarrow0}|f(x)|/|g(x)|=0$. According to (\ref{defn-BMAP}),
the BMAP is characterized by $\{\vc{D}(k);k\in\bbZ_+\}$ and thus is
denoted by BMAP $\{\vc{D}(k);k\in\bbZ_+\}$.

It should be noted that the infinitesimal generator of the background
Markov chain $\{J(t);t\ge0\}$ is given by $\vc{D} :=
\sum_{k=0}^{\infty} \vc{D}(k)$, which is irreducible and
conservative. We define $\vc{\eta}$ as the stationary distribution
vector of $\vc{D}$ and define $\lambda$ as the arrival rate of BMAP
$\{\vc{D}(k);k \in \bbZ_+\}$, i.e.,
\begin{equation}
\lambda =
\vc{\eta}\sum_{k=1}^{\infty}k\vc{D}(k)\vc{e}.
\label{defn-lambda}
\end{equation}
To avoid triviality, we assume that $\lambda \in (0,\infty)$. 

Let $\widehat{\vc{D}}(z) = \sum_{k=0}^{\infty}z^k\vc{D}(k)$
and 
\[
r_D 
= \sup\left\{
z\ge0; \sum_{k=0}^{\infty}z^k\vc{D}(k)~\mbox{is finite}
\right\}.
\]
We then assume the following.
\begin{assumpt}\label{assumpt-D}
$r_D > 1$.
\end{assumpt}

\begin{rem}
Assumption~\ref{assumpt-D} holds if $\{\vc{D}(k);k\in\bbN\}$ is
light-tailed, i.e., $\vc{D}(k) \le r^{-k} \vc{\varLambda}$ for some $r
> 1$ and $d \times d$ finite nonnegative matrix $\vc{\varLambda}$.
\end{rem}

For further discussion, we define $\wh{\vc{E}}(z)$, $0 \le z < r_D$, as
\begin{equation}
\wh{\vc{E}}(z) 
= \vc{I} +
{ \widehat{\vc{D}}(z) \over \dm\max_{j\in\bbD}|D_{j,j}(0)| }
\ge \vc{O}.
\label{defn-E(z)}
\end{equation}
It follows from \cite[Theorem~8.3.1]{Horn90} that there exists
nonnegative vectors $\vc{\eta}(z)=(\eta(z,j))_{j\in\bbD}$ and
$\vc{u}(z)=(u(z,j))_{j\in\bbD}$ such that, for $0 \le z < r_D$,
\begin{align}
&&&&
\vc{\eta}(z)\widehat{\vc{E}}(z) &= \delta_E(z) \vc{\eta}(z),
&
\widehat{\vc{E}}(z)\vc{u}(z) &= \delta_E(z) \vc{u}(z),&&&&
\label{defn-u_E(z)}
\\
&&&&
\vc{\eta}(z)\vc{u}(z) &= 1, & \vc{u}(z) &\ge \vc{e},&&&&
\label{eqn-eta(z)*u(z)}
\end{align}
where $\delta_E(z)$ denotes the spectral radius of
$\wh{\vc{E}}(z)$. Since $\vc{D}=\widehat{\vc{D}}(1)$ is an irreducible
infinitesimal generator, the nonnegative matrix $\wh{\vc{E}}(z)$ is
also irreducible for all $0 < z < r_D$, which implies that, for $0 < z
< r_D$, $\delta_E(z)$ is the Perron-Frobenius eigenvalue of
$\wh{\vc{E}}(z)$ \cite[Theorem~8.4.4]{Horn90}.

We now define
$\delta_D(z)$, $0 \le z < r_D$, as
\[
\delta_D(z) = (\delta_E(z) - 1) \max_{j\in\bbD}|D_{j,j}(0)|,
\]
where $\delta_D(z)$ is increasing and convex because so is $\delta_E(z)$
\cite{King61}. It follows from
(\ref{defn-E(z)}) and (\ref{defn-u_E(z)}) that
\begin{eqnarray}
\vc{\eta}(z)\widehat{\vc{D}}(z) = \delta_D(z) \vc{\eta}(z),
\qquad
\widehat{\vc{D}}(z)\vc{u}(z) = \delta_D(z) \vc{u}(z).
\label{defn-u_D(z)}
\end{eqnarray}
From (\ref{eqn-eta(z)*u(z)}) and (\ref{defn-u_D(z)}), we have
\begin{equation}
\delta_D(z) = \vc{\eta}(z)\widehat{\vc{D}}(z)\vc{u}(z).
\label{eqn-delta_D(z)}
\end{equation}
Note here that $\delta_D(z)$, $0 < z < r_D$, is a simple
eigenvalue of $\widehat{\vc{D}}(z)$. In addition,
Assumption~\ref{assumpt-D} shows that $\widehat{\vc{D}}(z)$ is {\it
  analytic} in a neighborhood of the point $z=1$. Therefore,
$\delta_D(z)$, $\vc{\eta}(z)$ and $\vc{u}(z)$ are analytic at $z=1$
\cite[Theorem 2.1]{Andr93}. Differentiating (\ref{eqn-delta_D(z)}) at
$z=1$ and using $\vc{\eta}(1) = \vc{\eta}$ and $\vc{u} = \vc{e}$, we
obtain
\begin{equation}
\delta_D'(1) = \vc{\eta}\widehat{\vc{D}}'(1)\vc{e}
= \lambda,
\label{eqn-delta_D'(1)}
\end{equation}
where the last equality holds due to (\ref{defn-lambda}). Note also
that $\delta_D(1) = 0$ because $\vc{D}$ is an irreducible and
conservative infinitesimal generator.

Next, we explain the queueing model considered in the section. The system consists of an infinite buffer and a possibly infinite number
of servers (the number of servers is not specified for
flexibility). Customers arrive at the system according to BMAP
$\{\vc{D}(k);k\in\bbZ_+\}$. When there are $k$ customers in the system
at time $t$, one of them leaves the system, independently of the other
customers, in time interval $(t,t+\Delta t]$ with probability $\mu(k)
\Delta t + o(\Delta t)$, where $\mu(0) = 0$ and $\mu(k) \ge 0$ for $k
\in \bbN$. Note here that the departure of a customer is caused by the
completion of its service or the impatience with waiting for the
service. In addition, the system can suffer from disasters, which can
be regarded as {\it negative customers} that remove all the customers
in the system including themselves on their arrivals. Disasters arrive
at the system according to a Poisson process with rate $\psi \ge 0$,
which is independent of the arrival and departure processes of
(ordinary) customers. We now define $L(t)$, $t \ge 0$, as the queue
length, i.e., the number of customers in the systems at time $t$. It
then follows that the joint process $\{(L(t),J(t));t\ge0\}$ of the
queue length and the background state is a continuous-time Markov
chain with state space $\bbF$ and infinitesimal generator
$\vc{Q}=(q(k,i;\ell,j))_{(k,i),(\ell,j)\in\bbF}$ given by
\begin{equation}
\vc{Q}
=
\left(
\begin{array}{ccccc}
\vc{D}(0) & \vc{D}(1) & \vc{D}(2) & \vc{D}(3) & \cdots
\\
(\psi + \mu(1))\vc{I} & \vc{D}(0) - \wt{\mu}(1)\vc{I} & \vc{D}(1) & \vc{D}(2) & \cdots
\\
\psi\vc{I} & \mu(2)\vc{I} & \vc{D}(0) - \wt{\mu}(2)\vc{I} & \vc{D}(1) & \cdots
\\
\psi\vc{I} & \vc{O} & \mu(3)\vc{I} & \vc{D}(0) - \wt{\mu}(3)\vc{I} & \cdots
\\
\vdots & \vdots & \vdots & \vdots & \ddots
\end{array}
\right),
\label{MG1-type-Q}
\end{equation}
where $\wt{\mu}(k)= \psi + \mu(k)$ for $k \in \bbN$.
It is easy to see that $\vc{Q} \in \sfBM_d$.

\begin{rem}
Suppose that $\psi = 0$ and $\wt{\mu}(k) = \mu(k) = \mu+(k-1)\mu'$ for
$k\in\bbN$, where $\mu,\mu' \in (0,\infty)$. In this case, $\vc{Q}$ is the
infinitesimal generator of the joint process of a BMAP/M/1 queue with
impatient customers and no disasters, where service times follow an
exponential distribution with rate $\mu$ and ``patient times" in queue
are independent and identically distributed (i.i.d.) exponentially
with rate $\mu'$. In addition, if $\mu=\mu'$, then $\vc{Q}$ is the
infinitesimal generator of the joint process of a BMAP/M/$\infty$
queue with service rate $\mu$.
\end{rem}

In what follows, we consider two cases: (a) no disasters occur; and
(b) disasters can occur.

\subsection{Case where no disasters occur}\label{subsec-example-1}

We make the following assumption.
\begin{assumpt}\label{assumpt-example-1}
(i) $\vc{Q}$ is irreducible; (ii) $\psi = 0$; and (iii)
  $\inf_{k\in\bbN}\mu(k) > \lambda$.
\end{assumpt}

Let $G(z)$, $0 < z < r_D$, denote
\[
G(z) = \inf_{k\in\bbN}\mu(k) \cdot (1 - z^{-1}) - \delta_D(z),
\qquad 0 < z < r_D.
\]
Recall here that $\delta_D(1) = 0$, $\delta_D'(1) = \lambda$ (see
(\ref{eqn-delta_D'(1)})) and $\delta_D(z)$ is increasing and convex
for $z \in (0, r_D)$. It follows from these facts and
Assumption~\ref{assumpt-example-1} that $G(z)$ is continuous for $z \in (0, r_D)$ and
\begin{eqnarray*}
G(1) &=& 0,\qquad G'(1) = \inf_{k\in\bbN}\mu(k) - \lambda > 0,
\end{eqnarray*}
which show that there exists some $\beta > 1$ such
that
\begin{equation}
c :=
\inf_{k\in\bbN}\mu(k) (1 - \beta^{-1}) - \delta_D(\beta) > 0.
\label{defn-C_{ast}}
\end{equation}

Let $\vc{Q}(k;\ell)=(q(k,i;\ell,j))_{i,j\in\bbD}$ for $k,\ell\in\bbZ_+$. 
Let $\vc{v}(k) = \beta^k \vc{u}(\beta)$ for $k \in \bbZ_+$. From
(\ref{defn-u_D(z)}), (\ref{MG1-type-Q}) and $\psi = 0$, we then have
\begin{eqnarray}
\sum_{\ell=0}^{\infty}\vc{Q}(0;\ell)\vc{v}(\ell)
&=& 
\sum_{\ell=0}^{\infty}\vc{D}(\ell)\vc{v}(\ell) 
\nonumber
\\
&=& \widehat{\vc{D}}(\beta) \vc{u}(\beta) = \delta_D(\beta) \vc{u}(\beta)
\nonumber
\\
&=& -c\vc{v}(0) + (c + \delta_D(\beta))\vc{u}(\beta)
\nonumber
\\
&\le& -c\vc{v}(0) + b\vc{e},
\label{eqn-44}
\end{eqnarray}
where 
\begin{equation}
b = (c + \delta_D(\beta))\max_{j\in\bbD}u(\beta,j). 
\label{defn-b-02}
\end{equation}
In addition, from (\ref{defn-u_D(z)}), (\ref{defn-C_{ast}}), (\ref{MG1-type-Q})
and $\psi = 0$, we have, for $k \in \bbN$,
\begin{eqnarray}
\sum_{\ell=0}^{\infty}\vc{Q}(k;\ell)\vc{v}(\ell)
&=& 
\sum_{\ell=0}^{\infty}\vc{D}(\ell)\vc{v}(l+k) 
+ \mu(k)\vc{v}(k-1) - \mu(k)\vc{v}(k)
\nonumber
\\
&=& \widehat{\vc{D}}(\beta) \beta^k\vc{u}(\beta)
- \mu(k)(1 - \beta^{-1}) \beta^k\vc{u}(\beta)
\nonumber
\\
&=& [\delta_D(\beta) 
- \mu(k) (1 - \beta^{-1})] \beta^k\vc{u}(\beta)
\nonumber
\\
&\le&  -c \vc{v}(k).
\label{eqn-43}
\end{eqnarray}
The inequalities (\ref{eqn-44}) and (\ref{eqn-43}) imply that
Assumptions~\ref{assumpt-basic} and \ref{assumpt-geo} hold for
$\wt{\vc{Q}}:=\vc{Q}$, $c \in (0,\infty)$ and $b \in (0,\infty)$ given
by (\ref{MG1-type-Q}), (\ref{defn-C_{ast}}) and (\ref{defn-b-02}),
respectively. As a result, it follows from
Theorem~\ref{thm-continu-geo} and Remark~\ref{rem-minimum-bound} that
\begin{eqnarray*}
\left\| \presub{n}\vc{\pi}_n - \vc{\pi} \right\|
&\le& {4b \over c} 
(t_1^{\ast}(n) + 1 ) \exp\{-t_1^{\ast}(n)\},\qquad n \in \bbN,
\end{eqnarray*}
where
\begin{equation*}
t_1^{\ast}(n)
= 
\max\left\{
-
\log\left( 
{1 \over 2c} 
\sum_{j\in\bbD}{\mu(n) + | D_{j,j}(0)| \over u(\beta,j)} \beta^{-n}
\right),
0 \right\},
\qquad n \in \bbN.
\end{equation*}

\subsection{Case where disasters can occur}

Instead of Assumption~\ref{assumpt-example-1}, we assume the
following.
\begin{assumpt}\label{assumpt-example-2}
(i) $\vc{Q}$ is irreducible; (ii) $\psi > 0$; and there exists some $K
  \in \bbZ_+$ such that
\begin{equation}
c' := 
\inf_{k \ge K+1}
\left\{ 
\mu(k) (1 - \beta^{-1}) + \psi(1 - \beta^{-k}) - \delta_D(\beta) 
\right\} > 0.
\label{defn-example-c}
\end{equation}
\end{assumpt}

\begin{rem}
Assumption~\ref{assumpt-example-2} holds if
\[
\liminf_{k\to\infty} \mu(k) 
> {\delta_D(\beta) - \psi\over 1 - \beta^{-1}}.
\]
\end{rem}

Let $\vc{v}'(k) = \beta^k
\vc{u}(\beta)$ for $k \in \bbZ_+$ and
\begin{eqnarray}
b' &=& \max_{0 \le k \le K}
\left\{
[c' + \delta_D(\beta) 
- \mu(k) (1 - \beta^{-1}) - \psi(1 - \beta^{-k})]\beta^k 
\right\}
\nonumber
\\
&& {} \quad \times 
\max_{j\in\bbD}u(\beta,j).
\label{defn-example-b}
\end{eqnarray}
Proceeding as in the derivation of (\ref{eqn-44}) and (\ref{eqn-43}),
we have
\begin{eqnarray}
\sum_{\ell=0}^{\infty}\vc{Q}(0;\ell)\vc{v}'(\ell)
&=& -c'\vc{v}'(0) + (c' + \delta_D(\beta))\vc{u}(\beta),
\label{eqn-44b}
\end{eqnarray}
and, for $k \in \bbN$,
\begin{eqnarray}
\lefteqn{
\sum_{\ell=0}^{\infty}\vc{Q}(k;\ell)\vc{v}'(\ell)
}
\quad
\nonumber
\\
&=& [\delta_D(\beta) 
- \mu(k) (1 - \beta^{-1})] \beta^k\vc{u}(\beta)
+ \psi\vc{v}'(0) - \psi\vc{v}'(k)
\nonumber
\\
&=& [\delta_D(\beta) 
- \mu(k) (1 - \beta^{-1}) - \psi(1 - \beta^{-k})]\beta^k\vc{u}(\beta)
\nonumber
\\
&=& -c'\vc{v}'(k) 
+ [c' + \delta_D(\beta) 
- \mu(k) (1 - \beta^{-1}) - \psi(1 - \beta^{-k})]\beta^k\vc{u}(\beta).
\qquad
\label{eqn-43b}
\end{eqnarray}
Applying (\ref{defn-example-c}) and (\ref{defn-example-b}) to
(\ref{eqn-44b}) and (\ref{eqn-43b}) yields
\begin{align}
\sum_{\ell=0}^{\infty}\vc{Q}(k;\ell)\vc{v}'(\ell)
&\le -c'\vc{v}'(k) + b' \vc{e}, & k &=0,1,\dots,K,
\label{add-eqm-151105-01}
\\
\sum_{\ell=0}^{\infty}\vc{Q}(k;\ell)\vc{v}'(\ell)
&\le -c'\vc{v}'(k), & k &=K+1,K+2,\dots.
\label{add-eqm-151105-02}
\end{align}
It follows from (\ref{add-eqm-151105-01}) and
(\ref{add-eqm-151105-02}) that (\ref{ineqn-hat{Q}v'}) in
Corollary~\ref{thm-continu-02} holds for $\wt{\vc{Q}}:=\vc{Q}$, $c'
\in (0,\infty)$ and $b' \in (0,\infty)$ given by (\ref{MG1-type-Q}),
(\ref{defn-example-c}) and (\ref{defn-example-b}), respectively, where
$K \in \bbZ_+$ is fixed such that (\ref{defn-example-c}) holds.  Note
here that $\wt{\vc{Q}}(K,0)\vc{e} = \vc{Q}(K,0)\vc{e} = \psi\vc{e} >
\vc{0}$. Thus, according to (\ref{add-defn-B}), fix $B =
b'\psi^{-1}$. Furthermore, according to
(\ref{add-defn-c}), (\ref{add-defn-b}) and (\ref{add-defn-v}), fix
\begin{eqnarray*}
c &=& {c' \over 1 + b'\psi^{-1}},
\label{add-defn-c-example}
\\
b &=& b' \left(1 - \psi^{-1} \min_{i\in\bbD}\sum_{j\in\bbD}D_{i,j}(0) \right),
\label{add-defn-b-example}
\\
v(k,i) &=& 
\left\{
\begin{array}{ll}
v'(0,i), & k=0,~i\in\bbD,
\\
v'(k,i) + b'\psi^{-1}, & k\in\bbN,~i\in\bbD.
\end{array}
\right.
\label{add-defn-v-example}
\end{eqnarray*}
Note also that 
\begin{eqnarray*}
\wt{q}(n,j;n,j) 
&=& q(n,j;n,j) = -\wt{\mu}(n) + D_{j,j}(0)
\nonumber
\\
&=& -\psi - \mu(n) - | D_{j,j}(0) |.
\end{eqnarray*}
Consequently, if follows from Corollary~\ref{thm-continu-02} that, for
$n \in \bbN$,
\begin{eqnarray*}
\left\| \presub{n}\vc{\pi}_n - \vc{\pi} \right\|
&\le& 
{4b' (1 + b'\psi^{-1}) \over c'}
\left(1 - \psi^{-1} \min_{i\in\bbD}\sum_{j\in\bbD}D_{i,j}(0) \right)
\nonumber
\\
&& {} \times 
(t_2^{\ast}(n) + 1 ) \exp\{-t_2^{\ast}(n)\},
\end{eqnarray*}
where $t_2^{\ast}(n)$, $n \in \bbN$, is given by
\begin{equation*}
t_2^{\ast}(n)
= \max\left\{
-
\log\left( 
{1 + b'\psi^{-1} \over 2c'} 
\sum_{j\in\bbD}
{
\psi + \mu(n) + |D_{j,j}(0)| \over u(\beta,j) + b' \psi^{-1} \beta^{-n}
} 
\beta^{-n}
\right),0
\right\}. 
\end{equation*}

\appendix

\section{Pathwise ordering}\label{sec-pathwise-ordering}

This appendix presents two lemmas on the pathwise ordering associated
with BMMCs. For this purpose, we consider two Markov chains
$\{(X_t,J_t);t\ge0\}$ and $\{(\wt{X}_t,\wt{J}_t);t\ge0\}$ with
infinitesimal generators
$\vc{Q}=(q(k,i;\ell,j))_{(k,i),(\ell,j)\in\bbF^{\leqslant N}}$ and
$\wt{\vc{Q}}=(\wt{q}(k,i;\ell,j))_{(k,i),(\ell,j)\in\bbF^{\leqslant
    N}}$, respectively, which have the same state space
$\bbF^{\leqslant N}$.  In what follows, we do not necessarily assume
that $\vc{Q}$ and $\wt{\vc{Q}}$ are ergodic.

\begin{lem}[Pathwise ordering of a BMMC]\label{lem1-continuous-ordering}
If $\vc{Q}$ is regular and $\vc{Q} \in \sfBM_d$, then there exist two
regular-jump Markov chains $\{(X'_t,J'_t);t\ge0\}$ and
$\{(X''_t,J''_t);t\ge0\}$ with infinitesimal generator $\vc{Q}$ on a
common probability space $(\Omega, \calF, \PP)$ such that
\begin{equation*}
X'_t(\omega) \le X''_t(\omega), \quad J'_t(\omega) = J''_t(\omega) 
\quad \mbox{for all}~ t > 0,
\end{equation*}
for any $\omega \in \Omega$ with $ X'_0(\omega) \le X''_0(\omega)$ and
$J'_0(\omega) = J''_0(\omega)$.
\end{lem}

\noindent
\begin{proof}
We fix $\delta > 0$ arbitrarily.  It follows from Lemma~\ref{lem-Q}
that $\vc{P}^{(\delta)} \in \sfBM$. Therefore, according to
\cite[Lemma~A.1]{Masu15-ADV}, we can construct two discrete-time
Markov chains $\{(Y_{\delta,\nu}',H_{\delta,\nu}');\nu\in\bbZ_+\}$ and
$\{(Y_{\delta,\nu}'',H_{\delta,\nu}'');\nu\in\bbZ_+\}$ with transition
probability matrix $\vc{P}^{(\delta)}$ on a probability space
$(\Omega, \calF, \PP)$ such that, for any $\omega \in \Omega$ with
$Y'_{\delta,0}(\omega) \le Y''_{\delta,0}(\omega)$ and
$H'_{\delta,0}(\omega) = H''_{\delta,0}(\omega)$,
\[
Y'_{\delta,\nu}(\omega) \le Y''_{\delta,\nu}(\omega),\quad 
H'_{\delta,\nu}(\omega) = H''_{\delta,\nu}(\omega)\quad 
\mbox{for all $\nu \in \bbN$}.
\]
Using the two Markov chains, we define two stochastic processes
$\{(X'_{\delta,t},J'_{\delta,t});t\ge0\}$ and
$\{(X''_{\delta,t},J''_{\delta,t});t\ge0\}$ on the probability space
$(\Omega, \calF, \PP)$ as follows:
\begin{align*}
&&&&
X'_{\delta,t} &= Y'_{\delta,\nu}, &
J'_{\delta,t} &= H'_{\delta,\nu}, & 
\nu\delta &\le t < (\nu+1)\delta,~\nu\in\bbZ_+,&&&&
\\
&&&&
X''_{\delta,t} &= Y''_{\delta,\nu}, &
J''_{\delta,t} &= H''_{\delta,\nu}, & 
\nu\delta &\le t < (\nu+1)\delta,~\nu\in\bbZ_+.&&&&
\end{align*}
It then holds that, for any $\omega \in \Omega$ with
$X'_{\delta,0}(\omega) \le X''_{\delta,0}(\omega)$ and
$J'_{\delta,0}(\omega) = J''_{\delta,0}(\omega)$,
\begin{equation}
X'_{\delta,t}(\omega) \le X''_{\delta,t}(\omega),\quad 
J'_{\delta,t}(\omega) = J''_{\delta,t}(\omega)\quad \mbox{for all $t > 0$}.
\label{pathwise-order-01}
\end{equation}
We also define 
\[
\calG_{\delta,s}' = \bigcup_{\alpha \ge \delta}\calF_{\alpha,s}',
\quad
\calG_{\delta,s}'' = \bigcup_{\alpha \ge \delta}\calF_{\alpha,s}'',
\quad \delta > 0,
\]
where $\calF_{\alpha,s}'$ and $\calF_{\alpha,s}''$, $\alpha > 0$, are
the $\sigma$-algebras generated by $\{(X'_{\alpha,u}, J'_{\alpha,u});0
\le u \le s\}$ and $\{(X''_{\alpha,u}, J''_{\alpha,u});0 \le u \le
s\}$, respectively.  Note here that
$\{(X'_{\delta,t},J'_{\delta,t});t\ge0\}$ is a semi-Markov process
with the embedded Markov chain
$\{(X_{\delta,\nu\delta}',J_{\delta,\nu\delta}')=(Y_{\delta,\nu}',H_{\delta,\nu}');\nu\in\bbZ_+\}$. Thus,
for $s \in [\nu\delta, (\nu+1)\delta)$, $t \in [(\nu+1)\delta,
    (\nu+2)\delta)$ and $\nu\in\bbZ_+$,
\begin{eqnarray}
\lefteqn{
\PP( X'_{\delta,t} = \ell, J'_{\delta,t} = j 
\mid X'_{\delta,s} = k, J'_{\delta,s} = i, \calG_{\delta,s}')
}
\quad&&
\nonumber
\\
&=& \PP( X'_{\delta,t} = \ell, J'_{\delta,t} = j 
\mid     X'_{\delta,s} = k, J'_{\delta,s} = i)
\nonumber
\\
&=& \PP( Y'_{\delta,\nu+1} = \ell, H'_{\delta,\nu+1} = j 
\mid     Y'_{\delta,\nu} = k, H'_{\delta,\nu} = i)
\nonumber
\\
&=& p^{(\delta)}(k,i;\ell,j), \qquad  (k,i;\ell,j) \in\bbF^2.
\label{transition-prob-(X',J')}
\end{eqnarray}
It follows from (\ref{transition-prob-(X',J')}) and (\ref{CK-EQ})
that, for all $t \ge s \ge 0$,
\begin{eqnarray}
\lefteqn{
\PP( X'_{\delta,t} = \ell, J'_{\delta,t} = j 
\mid X'_{\delta,s} = k, J'_{\delta,s} = i, \calG_{\delta,s}')
}
\quad&&
\nonumber
\\
&=& p^{( \{n_{\delta,t} -n_{\delta,s}\} \delta )}(k,i;\ell,j),
\qquad (k,i;\ell,j) \in\bbF^2,
\label{eqn-MC'-law}
\end{eqnarray}
where $n_{\delta,u} = \sup\{n\in\bbZ_+; n\delta \le u\}$ for $u \ge 0$.
Similarly, for all $t \ge s \ge 0$,
\begin{eqnarray}
\lefteqn{
\PP( X''_{\delta,t} = \ell, J''_{\delta,t} = j 
\mid X''_{\delta,s} = k, J''_{\delta,s} = i, \calG_{\delta,s}'')
}
\quad&&
\nonumber
\\
&=& p^{( \{n_{\delta,t} - n_{\delta,s}\} \delta )}(k,i;\ell,j),
\qquad (k,i;\ell,j) \in\bbF^2.
\label{eqn-MC''-law}
\end{eqnarray}

We now define
\begin{eqnarray*}
\{(X'_t,J'_t);t\ge0\} 
&=&
\lim_{\delta \downarrow0}\{(X'_{\delta,t},J'_{\delta,t});t\ge0\},
\\
\{(X''_t,J''_t);t\ge0\} 
&=& \lim_{\delta
  \downarrow0}\{(X''_{\delta,t},J''_{\delta,t});t\ge0\}.
\end{eqnarray*}
It then follows from (\ref{pathwise-order-01}) that, for any $\omega
\in \Omega$ with $X'_0(\omega) \le X''_0(\omega)$ and $J'_0(\omega) =
J''_0(\omega)$,
\[
X'_t(\omega) \le X''_t(\omega),\quad 
J'_t(\omega) = J''_t(\omega)\quad \mbox{for all $t > 0$}.
\]
Recall here that $\vc{Q}$ is regular and thus $\lim_{\delta \downarrow
  0}\vc{P}^{(\delta)} = \vc{I}$ (see
subsection~\ref{subsec-BSMC}). Therefore, we have
\begin{equation}
\lim_{\delta \downarrow 0}\vc{P}^{( \{n_{\delta,t} -n_{\delta,s}\} \delta )} =
\vc{P}^{(t-s)},\qquad 0 \le s \le t.
\label{add-eqn-151113-01}
\end{equation}
It follows from (\ref{eqn-MC'-law})--(\ref{add-eqn-151113-01}) and the
continuity of probability \cite[Chapter 1, Theorem 1.1]{Brem99} that,
for all $t \ge s \ge 0$ and $(k,i;\ell,j) \in\bbF^2$,
\begin{eqnarray*}
\PP( X'_t = \ell, J'_t = j 
\mid X'_s = k, J'_s = i, \cup_{\delta>0}\calG_{\delta,s}')
&=& p^{(t-s)}(k,i;\ell,j),
\\
\PP( X''_t = \ell, J''_t = j 
\mid X''_s = k, J''_s = i, \cup_{\delta>0}\calG_{\delta,s}'')
&=& p^{(t-s)}(k,i;\ell,j).
\end{eqnarray*}
As a result, $\{(X'_t, J'_t)\}$ and $\{(X''_t,
J''_t)\}$ are regular-jump Markov chains characterized by the common
transition matrix function $\vc{P}^{(t)}$ with the regular
infinitesimal generator $\vc{Q}$. The proof is completed.
\end{proof}

\begin{lem}[Pathwise ordering from the block-wise dominance relation]
\label{lem2-continuous-ordering}

Suppose that $\vc{Q} \prec_{d} \wt{\vc{Q}}$ and either $\vc{Q} \in
\sfBM_d$ or $\wt{\vc{Q}} \in \sfBM_d$. Under these conditions, the
following are true:
\begin{enumerate}
\item If $N < \infty$, i.e., $\vc{Q}$ and $\wt{\vc{Q}}$ are finite
  infinitesimal generators, then there exist two regular-jump Markov
  chains $\{(X'_t,J'_t);t\ge0\}$ and $\{(\wt{X}'_t,\wt{J}'_t);t\ge0\}$
  with infinitesimal generators $\vc{Q}$ and $\wt{\vc{Q}}$,
  respectively, on a common probability space $(\PP,\calF,\Omega)$ such
  that
\begin{equation*}
X'_t(\omega) \le \wt{X}'_t(\omega),\quad 
J'_t(\omega) = \wt{J}'_t(\omega) \quad 
\mbox{for all $t > 0$},
\end{equation*}
for any $\omega \in \Omega$ with $ X'_0(\omega) \le \wt{X}'_0(\omega)$
and $J'_0(\omega) = \wt{J}'_0(\omega)$.
\item If $\wt{\vc{Q}}$ is regular, then statement (a) holds
  without $N < \infty$.
\end{enumerate}
\end{lem}

\noindent
{\it Proof of Lemma~\ref{lem2-continuous-ordering}.}~ We first prove
statement (a). Since $\vc{Q}$ and $\wt{\vc{Q}}$ are finite infinitesimal generators, we can readily show that
\[
\exp\{\vc{Q} t\} \prec_d \exp\{ \wt{\vc{Q}} t\},
\]
in a way similar to the derivation of
(\ref{ineqn-exp{Q_1^{<n}}-exp{Q_2^{<n}}}). Thus, we have $\vc{P}^{(t)}
\prec_d \wt{\vc{P}}^{(t)}$. Furthermore, since either $\vc{Q}
\in \sfBM_d$ or $\wt{\vc{Q}} \in \sfBM_d$, it follows from
Lemma~\ref{lem-Q} that either $\vc{P}^{(t)} \in \sfBM_d$ or
$\wt{\vc{P}}^{(t)} \in \sfBM_d$ for all $t \ge 0$.

We now fix $\delta > 0$ arbitrarily. According to
\cite[Lemma~A.2]{Masu15-ADV}, we can construct two discrete-time
Markov chains $\{(Y'_{\delta,\nu},H'_{\delta,\nu});\nu\in\bbZ_+\}$ and
$\{(\wt{Y}'_{\delta,\nu},\wt{H}'_{\delta,\nu});\nu\in\bbZ_+\}$ with
transition probability matrices $\vc{P}^{(\delta)}$ and
$\wt{\vc{P}}^{(\delta)}$, respectively, on the common probability space
$(\Omega, \calF, \PP)$ such that, for any $\omega \in \Omega$ with
$Y'_{\delta,0}(\omega) \le \wt{Y}'_{\delta,0}(\omega)$ and
$H'_{\delta,0}(\omega) = \wt{H}'_{\delta,0}(\omega)$,
\[
Y'_{\delta,\nu}(\omega) \le \wt{Y}'_{\delta,\nu}(\omega),\quad 
H'_{\delta,\nu}(\omega) = \wt{H}'_{\delta,\nu}(\omega)\quad 
\mbox{for all $\nu \in \bbN$}.
\]
We then define two stochastic processes
$\{(X'_{\delta,t},J'_{\delta,t});t\ge0\}$ and
$\{(\wt{X}'_{\delta,t},\wt{J}'_{\delta,t});t\ge0\}$ on
the probability space $(\Omega, \calF, \PP)$ as follows:
\begin{align*}
&&&&
X'_{\delta,t} &= Y'_{\delta,\nu}, &
J'_{\delta,t} &= H'_{\delta,\nu}, & 
\nu\delta &\le t < (\nu+1)\delta,~\nu\in\bbZ_+,&&&&
\\
&&&&
\wt{X}'_{\delta,t} &= \wt{Y}'_{\delta,\nu}, &
\wt{J}'_{\delta,t} &= \wt{H}'_{\delta,\nu}, & 
\nu\delta &\le t < (\nu+1)\delta,~\nu\in\bbZ_+.&&&&
\end{align*}
By definition, for any $\omega \in \Omega$ with $X'_{\delta,0}(\omega)
\le \wt{X}'_{\delta,0}(\omega)$ and $J'_{\delta,0}(\omega) =
\wt{J}'_{\delta,0}(\omega)$,
\[
X'_{\delta,t}(\omega) \le \wt{X}'_{\delta,t}(\omega),\quad 
J'_{\delta,t}(\omega) = \wt{J}'_{\delta,t}(\omega)\quad
\mbox{for all $t > 0$}.
\]

Let
\begin{eqnarray*}
\{(X'_t,J'_t);t\ge0\} 
&=& \lim_{\delta \downarrow0}
\{(X'_{\delta,t},J'_{\delta,t});t\ge0\},
\\
\{(\wt{X}'_t,\wt{J}'_t);t\ge0\} 
&=& \lim_{\delta \downarrow0}
\{(\wt{X}'_{\delta,t},\wt{J}'_{\delta,t});t\ge0\}.
\end{eqnarray*}
Clearly, for any $\omega \in \Omega$ with $X'_0(\omega) \le
\wt{X}'_0(\omega)$ and $J'_0(\omega) = \wt{J}'_0(\omega)$,
\[
X'_t(\omega) \le \wt{X}'_t(\omega),\quad 
J'_t(\omega) = \wt{J}'_t(\omega)\quad
\mbox{for all $t > 0$}.
\]
In addition, proceeding as in the proof of Lemma~\ref{lem1-continuous-ordering}, we
can readily prove that $\{(X'_t,J'_t)\}$ and
$\{(\wt{X}'_t,\wt{J}'_t)\}$ are regular-jump Markov chains with
infinitesimal generators $\vc{Q}$ and $\wt{\vc{Q}}$, respectively. As
a result, statement (a) is proved.

As for statement (b), it follows from Lemma~\ref{lem-Q-3}~(b) that
$\vc{P}^{(t)} \prec_d \wt{\vc{P}}^{(t)}$ for all $t \ge 0$. Therefore,
we can prove statement (b) in the same way as the proof of statement
(a). The details are omitted. \qed

\begin{rem}\label{rem-first-meeting-lasts-forever}
The pathwise-ordered continuous-time Markov chains that appear in Lemmas~\ref{lem1-continuous-ordering} and \ref{lem2-continuous-ordering} can be generated via their respective {\it skeletons} (see the proofs of the lemmas), which are constructed in the way described in the proofs of Lemmas A.1 and A.2 of \cite{Masu15-ADV}. As shown in those proofs, the pathwise-ordered discrete-time  Markov chains therein are defined by the update functions $F^{-1}(u \mid k,i,j)$ and $\wt{F}^{-1}(u \mid k,i,j)$ unique to the respective transition probability matrices, together with common sequences of i.i.d.\ uniform random variables. Therefore, any pair of such pathwise-ordered Markov chains with a common transition probability matrix has the {\it first-meeting-lasts-forever property}, that is, the pathwise-ordered Markov chains run together (i.e., their trajectories coincide) forever after their first meeting time. As a result, we can assume that the first-meeting-lasts-forever property holds for pathwise-ordered continuous-time Markov chains with a common infinitesimal generator, which originate from Lemmas~\ref{lem1-continuous-ordering} and \ref{lem2-continuous-ordering} in this paper.
\end{rem}

\section{Basic lemmas}\label{appendix-basic}

This appendix presents basic lemmas, which are used in
Section~\ref{sec-bounds}.
\begin{lem}\label{prop-supermartingale}
Suppose that $\wt{\vc{Q}}$ is ergodic. If Assumption~\ref{assumpt-geo}
holds, then
\[
\{M_t:=\rme^{ct}v(\wt{X}_t,\wt{J}_t)
\dd{I}_{\{\wt{\tau}_0 > t\}};t\ge0\}
\]
is supermartingale, where $\wt{\tau}_0 = \inf\{t \ge 0:
\wt{X}_t = 0\}$.
\end{lem}

\begin{proof}
Since $\{(\wt{X}_t, \wt{J}_t)\}$ is a time-homogeneous Markov chain,
it suffices to prove that $\EE_{(k,i)}[M_t] \le v(k,i)$ for all $t \ge
0$ and $(k,i)\in\bbF$. Note that if $(\wt{X}_0, \wt{J}_0) = (0,i)$
then $\wt{\tau}_0 = 0$ and thus
\[
\EE_{(0,i)}[M_t]
= 0 < 1 \le v(0,i),\qquad t \ge 0,\ i \in \bbD,
\]
where the last inequality is due to $\vc{v} \ge \vc{e}$. 

In what follows, we prove that
\begin{equation}
\EE_{(k,i)}[M_t] \le v(k,i),
\qquad t \ge 0,\ (k,i)\in\bbN \times \bbD.
\label{add-eqn-20}
\end{equation}
Let $\wt{t}_n = \inf\{t\ge0:\wt{X}_t\ge n\}$ for $n \in \bbN$. It then
follows from the ergodicity of $\wt{\vc{Q}}$ that
$\PP(\lim_{n\to\infty}\wt{t}_n = \infty) = 1$ and
$\{\wt{t}_n;n\in\bbN\}$ is a nondecreasing sequence. Thus, using the
monotone convergence theorem, we obtain
\begin{eqnarray}
\lefteqn{
\EE_{(k,i)}\left[
v(\wt{X}_t, \wt{J}_t)
\dd{I}_{\{\wt{\tau}_0 > t\}} 
\right]
}
\quad &&
\nonumber
\\
&=& \lim_{n\to\infty}
\EE_{(k,i)}\left[v(\wt{X}_t, \wt{J}_t)
\dd{I}_{\{\wt{\tau}_0 > t,\,\wt{t}_{n+1} > t\}} \right]
\nonumber
\\
&=& \lim_{n\to\infty}g_t^{[1,n]}(k,i),
\qquad t \ge 0,\ (k,i) \in \bbN \times \bbD.
\label{add-eqn-151022-04}
\end{eqnarray}
where, for $(k,i) \in \bbF^{[1,n]}:=\{1,2,\dots,n\} \times \bbD$,
\begin{equation*}
g_t^{[1,n]}(k,i)
=
\EE_{(k,i)}\left[ 
v(\wt{X}_t, \wt{J}_t)
\dd{I}_{\{\wt{\tau}_0 > t,\,\wt{t}_{n+1} > t\}} 
\right], \qquad t \ge 0.
\end{equation*}
It then follows that column vector $\vc{g}_t^{[1,n]}
:=(g_t^{[1,n]}(k,i))_{(k,i) \in \bbF^{[1,n]}}$ satisfies
\begin{equation}
\vc{g}_t^{[1,n]}
= \exp\{\wt{\vc{Q}}^{[1,n]} t\}\vc{v}^{[1,n]}, \qquad t \ge 0,
\label{add-eqn-22}
\end{equation}
where $\wt{\vc{Q}}^{[1,n]} = (\wt{q}(k,i;\ell,j))_{(k,i),(\ell,j)\in
  \bbF^{[1,n]}}$ and $\vc{v}^{[1,n]} = (v(k,i))_{(k,i) \in
  \bbF^{[1,n]}}$. Furthermore, it follows from (\ref{ineqn-hat{Q}v})
that $\wt{\vc{Q}}_+^{\leqslant n}\vc{v}^{[1,n]} \le -c
\vc{v}^{[1,n]}$. Using this inequality and (\ref{add-eqn-22}), we have
\[
\vc{g}_t^{[1,n]}
= \exp\{\wt{\vc{Q}}^{[1,n]} t\}\vc{v}^{[1,n]}
\le \rme^{-ct}\vc{v}^{[1,n]}, \qquad t \ge 0.
\]
Substituting this inequality into (\ref{add-eqn-151022-04}) yields
\[
\EE_{(k,i)}\left[
v(\wt{X}_t, \wt{J}_t)
\dd{I}_{\{\wt{\tau}_0 >   t\}} 
\right]
\le \rme^{-ct}v(k,i),
\qquad t \ge 0,\ (k,i) \in \bbN \times \bbD,
\]
which shows that (\ref{add-eqn-20}) holds.
\end{proof}

\begin{lem}\label{appen-lem-2}
If Assumption~\ref{assumpt-basic-zero} holds and $\vc{Q}$ is
irreducible, then, for all $n \in \bbN$, $t \ge 0$ and $(k,i) \in
\bbF^{\leqslant n}$,
\begin{eqnarray}
\left\|  
\presub{n}\vc{p}_n^{(t)}(k,i) - \vc{p}^{(t)}(k,i)
\right\|
&\le& 2\int_0^t
\sum_{(\ell,j)\in\bbF}\presub{n}p_n^{(u)}(k,i;\ell,j)
\nonumber
\\
&& {} \times 
\sum_{(\ell',j')\in\bbF^{>n}}|q(\ell,j;\ell',j')| \rmd u,
\label{eqn-14b}
\end{eqnarray}
where $\vc{p}^{(t)}(k,i) = (p^{(t)}(k,i;\ell,j))_{(\ell,j) \in \bbF}$
and $\presub{n}\vc{p}_n^{(t)}(k,i) =
(\presub{n}\vc{p}_n(k,i;\ell,j))_{(\ell,j) \in \bbF}$.
\end{lem}

\proof  For $n \in \bbN$, $u \ge 0$ and $(k,i) \in
\bbF$, let $f_n^{(u)}(k,i)$ denote
\begin{eqnarray}
f_n^{(u)}(k,i) 
&=& \| \presub{n}\vc{p}_n^{(u)}(k,i) - \vc{p}^{(u)}(k,i) \|
\nonumber
\\
&=& \sum_{(\ell,j)\in\bbF} 
| \presub{n}p_n^{(u)}(k,i;\ell,j) - p^{(u)}(k,i;\ell,j) |.
\label{eqn-f_n^{(u)}(k,i)}
\end{eqnarray}
From the Chapman-Kolmogorov equation, we have
\begin{eqnarray*}
\lefteqn{
\presub{n}\vc{P}_n^{(u+\delta)} - \vc{P}^{(u+\delta)}
}
\quad &&
\nonumber
\\
&=& \presub{n}\vc{P}_n^{(u)}\presub{n}\vc{P}_n^{(\delta)}
- \vc{P}^{(u)}\vc{P}^{(\delta)}
\nonumber
\\
&=& \presub{n}\vc{P}_n^{(u)} 
\left( \presub{n}\vc{P}_n^{(\delta)} - \vc{P}^{(\delta)} \right)
+ \left( \presub{n}\vc{P}_n^{(u)} - \vc{P}^{(u)} \right)\vc{P}^{(\delta)}. 
\end{eqnarray*}
Thus, we obtain, for $\delta > 0$ and $(k,i)\in
\bbF^{\leqslant n}$,
\begin{eqnarray}
\lefteqn{
f_n^{(u+\delta)}(k,i)
}
\quad &&
\nonumber
\\
&\le& \sum_{(\ell,j)\in\bbF} \presub{n}p_n^{(u)}(k,i;\ell,j) 
\left\|
\presub{n}\vc{p}_n^{(\delta)}(\ell,j) - \vc{p}^{(\delta)}(\ell,j)
\right\|
\nonumber
\\
&& {} \qquad
+ 
\sum_{(\ell,j)\in\bbF} 
| \presub{n}p_n^{(u)}(k,i;\ell,j) - p^{(u)}(k,i;\ell,j) |
\sum_{(\ell',j')\in\bbF} 
p^{(\delta)}(\ell,j;\ell',j')
\nonumber
\\
&=& \sum_{(\ell,j)\in\bbF^{\leqslant n}}\presub{n}p_n^{(u)}(k,i;\ell,j) 
\left\| \presub{n}\vc{p}_n^{(\delta)}(\ell,j) - \vc{p}^{(\delta)}(\ell,j) 
\right\|
+ f_n^{(u)}(k,i),
\label{eqn-23}
\end{eqnarray}
where the last equality holds due to (\ref{eqn-(n)p_n^{(u)}=0}),
(\ref{eqn-f_n^{(u)}(k,i)}) and $\sum_{(\ell',j')\in\bbF}
p^{(\delta)}(\ell,j;\ell',j')=1$. Note here that
\[
\presub{n}p_n ^{(0)}(\ell,j;\ell',j') = p^{(0)}(\ell,j;\ell',j')
= \chi_{(\ell,j)}(\ell',j'),\qquad (\ell,j;\ell',j') \in \bbF^2.
\]
This equation and the triangle inequality yield, for $(\ell,j) \in \bbF$,
\begin{eqnarray}
\lefteqn{
\left\| \presub{n}\vc{p}_n^{(\delta)}(\ell,j) - \vc{p}^{(\delta)}(\ell,j) \right\|
}
\quad 
\nonumber
\\
&\le& 
\left\| \vc{p}^{(\delta)}(\ell,j) - \vc{p}^{(0)}(\ell,j) \right\| 
+ \left\| \presub{n}\vc{p}_n^{(\delta)}(\ell,j) - \presub{n}\vc{p}_n^{(0)}(\ell,j) \right\|
\nonumber
\\
&=& 2\left\{ 1 - p^{(\delta)}(\ell,j;\ell,j) \right\}
+ 2 \left\{ 1 - \presub{n}p_n^{(\delta)}(\ell,j;\ell,j) \right\}.
\label{add-eqn-151112-01}
\end{eqnarray}
It follows from Assumption~\ref{assumpt-basic-zero} and
Lemma~\ref{lem-(n)Q_*-regular} that both $\vc{Q}$ and
$\presub{n}\vc{Q}_n$ are stable and conservative and their respective
transition matrix functions are continuous.  Thus, we have
\cite[Theorem II.3.1]{Chun67}
\begin{align*}
1 - p^{(\delta)}(\ell,j;\ell,j) &\le \delta \left| q(\ell,j;\ell,j) \right|,
& (\ell,j) & \in \bbF,
\\
1 - \presub{n}p_n^{(\delta)}(\ell,j;\ell,j) &\le \delta 
\left| \presub{n}q_n(\ell,j;\ell,j) \right|,
& (\ell,j) & \in \bbF.
\end{align*}
Applying these inequalities to (\ref{add-eqn-151112-01}) yields
\begin{eqnarray*}
\left\| \presub{n}\vc{p}_n^{(\delta)}(\ell,j) - \vc{p}^{(\delta)}(\ell,j) \right\|
&\le& 2\delta \cdot \left\{ | q(\ell,j;\ell,j) | +  |\presub{n}q_n(\ell,j;\ell,j)| \right\}
\nonumber
\\
&\le& 4\delta \cdot | q(\ell,j;\ell,j) |,
\end{eqnarray*}
where the last inequality follows from
(\ref{def-trunc{q}_n(k,i;l,j)}). Note here that
\begin{align*}
\lim_{\delta \downarrow 0}
{
\presub{n}p_n^{(\delta)}(\ell,j;\ell',j') - \chi_{(\ell,j)}(\ell',j')
\over \delta
}
&= \presub{n}q_n^{(\delta)}(\ell,j;\ell',j'),
& (\ell,j;\ell',j') &\in \bbF^2,
\\
\lim_{\delta \downarrow 0}
{
p^{(\delta)}(\ell,j;\ell',j') - \chi_{(\ell,j)}(\ell',j')
\over \delta
}
&= q^{(\delta)}(\ell,j;\ell',j'),
& (\ell,j;\ell',j') &\in \bbF^2.
\end{align*}
Therefore, using the dominated convergence theorem, we obtain
\begin{equation}
\lim_{\delta \downarrow 0}
{\left\| 
\presub{n}\vc{p}_n^{(\delta)}(\ell,j) - \vc{p}^{(\delta)}(\ell,j) 
\right\| 
\over \delta
}
= \left\| \presub{n}\vc{q}_n(\ell,j) - \vc{q}(\ell,j)  \right\|,
~~~(\ell,j) \in \bbF,
\label{add-eqn-24}
\end{equation}
where $\presub{n}\vc{q}_n(k,i) =
(\presub{n}q_n(k,i;\ell,j))_{(\ell,j)\in\bbF}$ and $\vc{q}(k,i) =
(q(k,i;\ell,j))_{(\ell,j)\in\bbF}$ for $(k,i)\in\bbF$. Combining
(\ref{eqn-23}) and (\ref{add-eqn-24}), we
have
\begin{eqnarray}
{\partial \over \partial u}f_n^{(u)}(k,i)
&\le& 
\sum_{(\ell,j)\in\bbF^{\leqslant n}}\presub{n}p_n^{(u)}(k,i;\ell,j) 
\left\| \presub{n}\vc{q}_n(\ell,j) - \vc{q}(\ell,j) \right\|.
\label{eqn-28}
\end{eqnarray}
From (\ref{def-trunc{q}_n(k,i;l,j)}), we also have
\[
\left\|
\presub{n}\vc{q}_n(\ell,j) - \vc{q}(\ell,j) \right\|
= 2\sum_{(\ell',j')\in\bbF^{>n}}|q(\ell,j;\ell',j')|,
\quad (\ell,j) \in \bbF^{\leqslant n}.
\]
Substituting this into (\ref{eqn-28}) and using
(\ref{eqn-(n)p_n^{(u)}=0}) yield, for $(k,i) \in \bbF^{\leqslant n}$,
\begin{eqnarray}
{\partial \over \partial u}f_n^{(u)}(k,i)
&\le& 2
\sum_{(\ell,j)\in\bbF^{\leqslant n}}\presub{n}p_n^{(u)}(k,i;\ell,j) 
\sum_{(\ell',j')\in\bbF^{>n}}|q(\ell,j;\ell',j')|
\nonumber
\\
&=& 2
\sum_{(\ell,j)\in\bbF}\presub{n}p_n^{(u)}(k,i;\ell,j) 
\sum_{(\ell',j')\in\bbF^{>n}}|q(\ell,j;\ell',j')|.
\label{add-eqn-28}
\end{eqnarray}
Note here that $f_n^{(0)}(k,i) = \sum_{(\ell,j) \in \bbF}
|\chi_{(k,i)}(\ell,j) - \chi_{(k,i)}(\ell,j)| = 0$. As a result,
integrating both sides of (\ref{add-eqn-28}) with respect to $u$ from
$0$ to $t$, we obtain (\ref{eqn-14b}).  \qed


\section*{Acknowledgments}
The author acknowledges stimulating discussions with Shusaku Sakaiya,
which led to Remark~\ref{rem-minimum-bound}. This research was
supported in part by JSPS KAKENHI Grant Number 15K00034.


\begin{thebibliography}{00}
\bibitem{Ande91}
W.J.~Anderson,
Continuous-Time Markov Chains: An Applications-Oriented Approach,
Springer, New York, 1991.

\bibitem{Andr93}
A.L.~Andrew, K.-W.E.~Chu, P.~Lancaster,
Derivatives of eigenvalues and eigenvectors of matrix functions,
SIAM Journal on Matrix Analysis and Applications 14 (1993) 903--926.

\bibitem{Bill99}
P.~Billingsley,
Convergence of Probability Measures, second ed.,
John Wiley \& Sons, New York, 1999.

\bibitem{Bini05}
D.A.~Bini, G.~Latouche, B.~Meini,
Numerical Methods for Structured Markov Chains,
Oxford University Press, New York, 2005.

\bibitem{Brem99}
P.~Br\'{e}maud,
Markov Chains: Gibbs Fields, Monte Carlo Simulation, and Queues,
Springer, New York, 1999.

\bibitem{Chun67}
K.L.~Chung,
Markov Chains with Stationary Transition Probabilities, second ed.,
Springer, Berlin, 1967.

\bibitem{Dale68}
D.J. Daley,
Stochastically monotone Markov chains,
Probability Theory and Related Fields 10 (1968) 305--317.

\bibitem{Gibs87}
D.~Gibson, E.~Seneta,
Monotone infinite stochastic matrices and their augmented truncations,
Stochastic Processes and their Applications  24 (1987) 287--292.

\bibitem{Hart12}
A.G.~Hart, R.L.~Tweedie,
Convergence of invariant measures of truncation approximations to Markov processes,
Applied Mathematics 3 (2012) 2205--2215.

\bibitem{Herv14}
L.~Herv\'{e}, J.~Ledoux,
Approximating Markov chains and V-geometric ergodicity via weak
  perturbation theory,
Stochastic Processes and their Applications 124 (2014) 613--638.

\bibitem{Horn90}
R.A.~Horn, C.R.~Johnson,
Matrix Analysis, paperback ed.,
Cambridge University Press, Cambridge, UK, 1990.

\bibitem{King61}
J.F.C.~Kingman,
A convexity property of positive matrices,
The Quarterly Journal of Mathematics 12 (1961) 283--284.

\bibitem{Kulk10}
V.G.~Kulkarni,
Modeling and Analysis of Stochastic Systems, second ed.,
CRC Press, Boca Raton, FL, 2010.

\bibitem{Lato99}
G.~Latouche, V.~Ramaswami,
Introduction to Matrix Analytic Methods in Stochastic Modeling,
ASA-SIAM, Philadelphia, PA, 1999.

\bibitem{LiHai00}
H.~Li, Y.Q.~Zhao,
Stochastic block-monotonicity in the approximation of the stationary
  distribution of infinite Markov chains, 
Stochastic Models 16 (2000) 313--333.

\bibitem{LiuYuan10}
Y.~Liu,
Augmented truncation approximations of discrete-time Markov chains,
Operations Research Letters 38 (2010) 218--222.

\bibitem{Luca91}
D.M.~Lucantoni,
New results on the single server queue with a batch Markovian
  arrival process,
Stochastic Models 7 (1991) 1--46.

\bibitem{Lund96}
R.B.~Lund, S.P.~Meyn, and R.L.~Tweedie,
Computable exponential convergence rates for stochastically ordered Markov 
processes, The Annals of Applied Probability 6 (1996) 218--237.

\bibitem{Masu15-ADV}
H.~Masuyama,
Error bounds for augmented truncations of discrete-time
  block-monotone Markov chains under geometric drift conditions,
Advances in Applied Probability 47 (2015) 83--105.

\bibitem{Masu16-SIAM}
H.~Masuyama,
Error bounds for augmented truncations of discrete-time
  block-monotone Markov chains under subgeometric drift conditions,
SIAM Journal on Matrix Analysis and Applications 37 (2016) 877--910. 

\bibitem{Meyn09}
S.~Meyn, R.L.~Tweedie,
Markov Chains and Stochastic Stability, second ed.,
Cambridge University Press, Cambridge, UK, 2009.

\bibitem{Mull02}
A.~M\"uller, D.~Stoyan,
Comparison Methods for Stochastic Models and Risks, 
John Wiley \& Sons, Chichester, UK, 2002.

\bibitem{Neut89}
M.F.~Neuts,
Structured Stochastic Matrices of M/G/1 Type and Their Applications, 
Marcel Dekker, New York, 1989.

\bibitem{Tijm03}
H.C.~Tijms,
A First Course in Stochastic Models,
John Wiley \& Sons, Chichester, UK, 2003.

\bibitem{Twee98}
R.L.~Tweedie,
Truncation approximations of invariant measures for Markov chains,
Journal of Applied Probability 35 (1998) 517--536.

\bibitem{Will91}
D.~Williams,
Probability with Martingales,
Cambridge University Press, Cambridge, UK, 1991.

\bibitem{Zeif14b}
A.~Zeifman, V.~Korolev, Y.~Satin, A.~Korotysheva, V.~Bening,
Perturbation bounds and truncations for a class of Markovian
  queues, Queueing Systems 76 (2014) 205--221.

\bibitem{Zeif12}
A.~Zeifman, A.~Korotysheva,
Perturbation bounds for {${\rm M}_t/{\rm M}_t/N$} queue with
  catastrophes,
Stochastic Models, 28 (2012) 49--62.

\bibitem{Zeif14a}
A.~Zeifman, Y.~Satin, V.~Korolev, S.~Shorgin,
On truncations for weakly ergodic inhomogeneous birth and death
  processes, 
International Journal of Applied Mathematics and Computer Science 
24 (2014) 503--518.
\end{thebibliography}


\end{document}